\documentclass[11pt, twoside, leqno]{article}

\usepackage{amsmath}
\usepackage{amssymb}
\usepackage{titletoc}
\usepackage{mathrsfs}
\usepackage{amsthm}
\usepackage{indentfirst}
\usepackage{color}
\usepackage{txfonts}

\allowdisplaybreaks

\def\bint{{\ifinner\rlap{\bf\kern.30em--}
\int\else\rlap{\bf\kern.35em--}\int\fi}\ignorespaces}

\def\sbint{{\ifinner\rlap{\bf\kern.32em--}
\hspace{0.078cm}\int\else\rlap{\bf\kern.45em--}\int\fi}\ignorespaces}

\def\rr{\mathbb{R}}
\def\rn{\mathbb{R}^n}

\def\cc{\mathbb{C}}
\def\nn{\mathbb{N}}

\def\lz{\lambda}

\def\ls{\lesssim}

\def\BMO{\mathrm{BMO}}
\def\morrey{L^{p,\,\lambda}}
\def\CMO{\mathrm{CMO}}

\def\r{\right}
\def\lf{\left}

\def\gfz{\genfrac{}{}{0pt}{}}

\def\r{\right}
\def\lf{\left}
\def\gfz{\genfrac{}{}{0pt}{}}

\def\supp{{\mathop\mathrm{\,supp\,}}}

\def\loc{{\mathop\mathrm{\,loc\,}}}

\def\BMO{{\mathop\mathrm{\,BMO\,}}}

\def\eqref#1{(\ref{#1})}

\def\func#1{\mathop{\mathrm{#1}}}

\def\supp{\func{supp}}

\newtheorem{theorem}{Theorem}[section]
\newtheorem{lemma}[theorem]{Lemma}

\theoremstyle{definition}

\newtheorem{definition}[theorem]{Definition}

\numberwithin{equation}{section}

\begin{document}

\title{\bf\Large Boundedness and Compactness Characterizations of Cauchy Integral Commutators on
Morrey Spaces
\footnotetext{\hspace{-0.35cm} 2010 {\it
Mathematics Subject Classification}: Primary 42B20;
Secondary 47B47, 46E35, 42B35. \endgraf
{\it Key words and phrases}. Cauchy integral, commutator, boundedness, compactness, Morrey space.
 \endgraf
Dachun Yang is supported by the NNSF of China (Grant No. 11571039 and and 11671185).
Dongyong Yang is supported by the NNSF of China (Grant No. 11571289).}}
\date{ }
\author{Jin Tao, Dachun Yang and Dongyong Yang\footnote{Corresponding author / January 4, 2018.}}
\maketitle

\vspace{-0.8cm}

\begin{center}
\begin{minipage}{13cm}
{\small {\bf Abstract}\quad Let $C_\Gamma$
be the Cauchy integral operator on a Lipschitz curve $\Gamma$. In this article,
the authors show that the commutator $[b,C_\Gamma]$ is bounded (resp., compact) on the
Morrey space $L^{p,\,\lambda}(\mathbb R)$  for any (or some)
$p\in(1, \infty)$ and $\lambda\in(0, 1)$ if and only if $b\in {\rm BMO}(\mathbb R)$
(resp., ${\rm CMO}(\mathbb R)$). As an application, a factorization of the classical
Hardy space $H^1(\mathbb R)$ in terms of $C_\Gamma$ and its adjoint operator  is obtained.}
\end{minipage}
\end{center}

\vspace{0.0cm}

\section{Introduction}

To study the factorization theorem of the Hardy space $H^1(\rn)$, in their celebrated
work \cite{CRW1976}, Coifman et al. proved that, if a function $b\in {\rm BMO}(\mathbb R^n)$, then the
commutator $[b,T]f:=bT (f)-T(bf)$  of a Calder\'on-Zygmund singular integral operator $T$ of convolution
type with smooth kernel is bounded on $L^p(\mathbb R^n)$ for any $p\in(1, \infty)$; they also proved that, if $[b, R_j]$ is bounded on $L^p(\mathbb R^n)$ for every Riesz transform $R_ j$, $j\in\{1,\ldots, n\}$, then $b\in {\rm BMO}(\mathbb R^n)$.
This equivalent characterization of the boundedness of commutators was further studied by
Janson \cite{j78} and Uchiyama \cite{U1978}, respectively.  Moreover, Uchiyama \cite{U1978}
showed that $[b, T]$ is compact on $L^p(\mathbb R^n)$ for any $p\in(1,\infty)$ if and only if
$b\in{\rm CMO}(\mathbb R^n)$, which is the BMO$(\rn)$-closure of $C_c^\infty(\rn)$,
the set of all infinitely differentiable functions on $\rn$ with compact supports.
Since then, there have been a lot of articles concerning the boundedness and the compactness
of commutators on function spaces as well as their applications in PDEs (see, for example, \cite{clms,is,CDW2009,ly,yy,t,GLW2017,MWY} and references therein).
In particular, Di Fazio and Ragusa \cite{dr} in 1991 gave a characterization of the boundedness
of $[b, T]$ on the Morrey space $L^{p,\,\lz}(\rn)$ for any $\lambda\in(0, n)$ and $p\in(1, \infty)$.
In 2012, Chen et al. \cite{CDW2012} further established the equivalent characterization of the
compactness of $[b, T]$ on $L^{p,\,\lz}(\rn)$ for any $\lambda\in(0,n)$ and $p\in(1,\infty)$.
For more results on the boundedness of operators on Morrey spaces, we refer the reader
to \cite{AM1997,KM2006, KS2009,AX2012,DX2007}.
We only mention that, since they were introduced by Morrey in \cite{m}, Morrey spaces have
proved very useful in PDEs;
see, for example, \cite{m,k,dr93,s} and references therein.

Let $\lambda\in(0, 1)$ and $C_\Gamma$ be the well-known Cauchy integral operator on a
Lipschitz curve $\Gamma$ (see Definition \ref{Cauchy} below). It is well known that
$C_\Gamma$ is a Calder\'{o}n-Zygmund operator of non-convolution type.
Recently, Li et al. \cite{LNWW2017} obtained the equivalent characterizations of
the boundedness and the compactness of the commutator $[b, C_\Gamma]$ on $L^p(\rr)$
for any $p\in(1, \infty)$. Observing that the Lebesgue space $L^p(\rr)$ can be seen as
$\morrey(\rr)$ with $\lambda=0$, the purpose of this article is to establish the equivalent
characterizations of the boundedness and the compactness of $[b, C_\Gamma]$ on
Morrey spaces $\morrey(\rr)$ for any $\lambda\in(0, 1)$ and $p\in(1,\infty)$.

In what follows, for any $q\in(0,\infty)$, we use $L^q(\rr)$ to denote the set of all measurable
functions $f$ such that
$$\|f\|_{L^q(\rr)}:=\lf\{ \int_\rr|f(x)|^q\,dx\r\}^\frac{1}{q}<\infty.$$
Meanwhile, we use $L^\infty(\rr)$ to denote the set of all essentially bounded functions,
equipped with the following norm: for any $g\in L^\infty(\rr)$,
$$\|g\|_{L^\infty(\rr)}:=\inf_{\gfz{E\subset\rr}{|E|=0}}\sup_{x\in\rr\setminus E}|f(x)|.$$

Let us start with recalling the definition of the Cauchy integral operator.
\begin{definition}\label{Cauchy}
  Let $A$ be a Lipschitz function on $\rr$ [i.\,e., $A'=:a\in L^\infty(\rr)$] and
  let $$\Gamma:=\lf\{(t,A(t)):\,t\in\rr\r\}$$ be a plane curve. The \emph{Cauchy integral} $C_\Gamma(f)$
  for suitable function $f$ on $\rr$ is defined by setting
\begin{align*}
C_\Gamma(f)(x):=\,&\mathrm{p.\,v.\,}\frac{1}{\pi i}\int_{\rr}\frac{f(y)}{y-x+i[A(y)-A(x)]}\,dy,\,
\forall x\in\rr
\end{align*}
\end{definition}
Obviously, the kernel of $C_\Gamma$, denoted still by $C_\Gamma$, satisfies that
there exists a positive constant $C$ such that, for any $x,\,y\in\rr$,
\begin{enumerate}
    \item[{\rm(i)}]$|C_\Gamma(x,y)|\le C\frac{1}{|x-y|}\quad\mathrm{if}\quad |x-y|>0$;
    \item[{\rm(ii)}]$|C_\Gamma(x,y)-C_\Gamma(x,z)|+|C_\Gamma(y,x)-C_\Gamma(z,x)|\le C\frac{|y-z|}{|x-y|^{2}}\quad
       \mathrm{if}\quad |x-y|>2|y-z|\ge0$.
\end{enumerate}
Recalling that $C_\Gamma$ is bounded on $L^2(\rr)$ (\cite{CMM1982}), we then know that
$C_\Gamma$ is a standard Calder\'{o}n-Zygmund operator.

Now we recall the notions of $\BMO(\rr)$ and $\morrey(\rr)$ as follows.
\begin{definition}\label{BMO}
A locally integrable real-valued function $f :\,\,\rr \to \rr$ is said to be of
\emph{bounded mean oscillation}, denoted as $f\in\BMO(\rr)$, if
$$\|f\|_{\BMO(\rr)}:=\sup_{I\subset\rr}M(f,I)<\infty,$$
where $I$ is any interval of $\rr$,
$$M(f,I):=\frac{1}{|I|}\int_I\lf|f(x)-f_I\r|\,dx,\quad f_I:=\frac{1}{I}\int_If(y)\,dy,$$
and the supremum is taken over all intervals $I\subset\rr$.
\end{definition}

\begin{definition}
  Let $p\in(1,\infty)$ and $\lambda\in(0,1)$. The \emph{Morrey space} $\morrey(\rr)$ is defined
  by setting
  $$\morrey(\rr):=\lf\{f\in L_\mathrm{loc}^p(\rr):\,\,\|f\|_{\morrey(\rr)}<\infty \r\},$$
  where
  $$\|f\|_{\morrey(\rr)}:=\sup_{\substack{x\in\rr\\r\in(0,\infty)}}\lf[\frac{1}{r^\lambda}
  \int_{I(x,\,r)}|f(y)|^p\,dy\r]^{1/p}$$
  and, for any $x\in\rr$ and $r\in(0,\infty)$,
the interval $I(x,r):=\{y\in\rr:\,\,|y-x|<r \}$.
\end{definition}

Our first main result reads as follows.

\begin{theorem}\label{thm1}
Let $b \in L_{\loc}^{1}(\rr)$.
Then the following statements are mutually equivalent:
\begin{itemize}
  \item [{\rm (i)}] $b\in \BMO(\rr)$.
  \item [{\rm (ii)}] The commutator $[b,C_{\Gamma}]$ is
bounded on $\morrey(\rr)$ for any $p \in (1,\infty)$ and $\lambda\in(0,1)$.
  \item [{\rm(iii)}] The commutator $[b,C_{\Gamma}]$ is
bounded on $\morrey(\rr)$ for some $p \in (1,\infty)$ and $\lambda\in(0,1)$.
\end{itemize}
Moreover, for any given $p \in (1,\infty)$ and $\lambda\in(0,1)$, there exist positive
constants $c$ and $C$, depending on $p$ and $\lambda$, such that
$$\|b\|_{\BMO(\rr)}/c\le \lf\|[b,C_{\Gamma}]\r\|_{\morrey(\rr)\to\morrey(\rr)}\le C\|b\|_{\BMO(\rr)},$$
here and hereafter, $\|[b,C_\Gamma]\|_{\morrey(\rr)\to\morrey(\rr)}$ denotes
the operator norm of $[b,C_\Gamma]$ on $\morrey(\rr)$.
\end{theorem}

Let $H^1(\rr)$ be the classical \emph{atomic Hardy space} (see, for example, \cite{Duo}). That is,
$$H^1(\rr):=\lf\{f\in L^1(\rr):\,\,f=\sum_{j=1}^\infty\lambda_j a_j,\,\,\{a_j\}_{j=1}^\infty
\,\,\mathrm{are\,\,atoms},\,\,\{\lambda_j\}_{j=1}^\infty\subset\mathbb{C}\,\,\mathrm{and}
\,\,\sum_{j=1}^\infty|\lambda_j|<\infty \r\}$$
equipped with the norm
$$\|f\|_{H^1(\rr)}:=\inf\lf\{\sum_{j=1}^\infty\lf|\lambda_j\r| \r\},$$
where the infimum is taken over all possible decompositions of $f$ as above.
Here an \emph{atom}  $a$ is a function in $L^1(\rr)$ which is supported in an interval $I:=I(x,r)$,
with $x\in\rr$ and $r\in(0,\infty)$, and satisfies that
$$\|a\|_{L^\infty(\rr)}\le r^{-1}\quad\mathrm{and}\quad\int_I a(x)\,dx=0.$$
As a corollary of Theorem \ref{thm1} and the well-known fact that $[H^1(\rr)]^\ast=\BMO(\rr)$ (see
Fefferman and Stein \cite{FS1972}), the second main result of this article concerns a factorization
of $H^1(\rr)$ via $C_\Gamma$ and $C_\Gamma^\ast$, here and hereafter, for any linear operator $T$,
$T^\ast$ denotes the \emph{adjoint operator} of $T$. We begin with recalling the notion of blocks
in \cite{BRV1999}.
For more properties of blocks and related spaces, see  \cite{L1984,BRV1999, KM2006, RT2015}.

\begin{definition}\label{block}
Let $q\in(1,\infty)$ and $\lambda\in(0,1)$. A function $b$ is called a \emph{$(\lambda,q)$-block} if there
exists an interval $I(x_0,r)$, with $x_0\in\rr$ and $r\in(0,\infty)$, such that
$$\supp(b)\subset I(x_0,r)\quad \mathrm{and}\quad \|b\|_{L^q(\rr)} \le |I(x_0,r)|^{-\frac{\lambda}{q'}},$$
where $1/q+1/q'=1$.
\end{definition}

We further recall the definition of $h^{\lambda,\ q}(\rr)$ via $(\lambda,\,q)$-blocks from \cite{BRV1999}.

\begin{definition}\label{hspace}
Let $q\in(1,\infty)$ and $\lambda\in(0,1)$. The \emph{space $h^{\lambda,\,q}(\rr)$} is defined by setting
\begin{align*}
h^{\lambda,\,q}(\rr):=&\left\{g\in L^1_{\rm loc}(\rr):\,\,g=\sum_{j=1}^\infty m_j u_j,
\ \{u_j\}_{j=1}^\infty\ \mathrm{are} \ (\lambda,q)\mathrm{-blocks},\right.\\
&\quad\quad\quad\lf.\{m_j\}_{j=1}^\infty\subset\cc\ \mathrm{and}\ \sum_{j=1}^\infty \lf|m_j\r|<\infty \right\}.
\end{align*}
Moreover, for any $g\in h^{\lambda,\ q}(\rr)$, let
$$\|g\|_{h^{\lambda,\,q}(\rr)}:=\inf \lf\{\sum_{j=1}^\infty\lf|m_j\r|\r\},$$
where the infimum is taken over  all possible decompositions of $g$ as above.
\end{definition}

It was showed in \cite{BRV1999} that $h^{\lambda,\,q}(\rr)$ is a Banach space and the dual space of
$h^{\lambda,\,q}(\rr)$ is $L^{q',\,\lz}(\rr)$; see also \cite{L1984,KM2006,RT2015}.
Now we state the factorization of $H^1(\rr)$ in terms of $h^{\lambda,\,q}(\rr)$ and $L^{q',\,\lz}(\rr)$
as follows; for the case of factorization in terms of Calder\'on-Zygmund operators $T$ of
convolution type and generalized Morrey spaces and their predual spaces on $\rn$,
we refer the reader to \cite[Theorem 3.2]{KM2006}. In what follows, we use $L_c^\infty(\rr)$ to
denote the set of all $L^\infty(\rr)$ functions with compact supports.

\begin{theorem}\label{fac}
 For any $f\in H^1(\rr)$, there exist $\{\lambda_k^l\}_{k,\,l\in\nn}\subset\mathbb{C}$ and
 functions $\{g_k^l\}_{k,\,l\in\nn},\,\,\{h_k^l\}_{k,\,l\in\nn}\subset L_c^{\infty}(\rr)$ such that
 $f=\sum_{l=1}^\infty\sum_{k=1}^\infty \lambda_k^l(g_k^l C_\Gamma^\ast h_k^l-h_k^l C_\Gamma g_k^l)$
 in $H^1(\rr)$ and
  $$\|f\|_{H^1(\rr)}\sim\inf\lf\{\sum_{l=1}^\infty\sum_{k=1}^\infty\lf|\lambda_k^l\r|\lf\|g_k^l\r\|_
  {\morrey(\rr)}\lf\|h_k^l\r\|_{h^{\lambda,\,p'}(\rr)}:\,\,f=\sum_{l=1}^\infty\sum_{k=1}^\infty \lambda_k^l\lf(g_k^l C_\Gamma^\ast h_k^l-h_k^l C_\Gamma g_k^l\r)\r\}$$
 with the equivalent positive constants independent of $f$.
\end{theorem}

Let $\CMO(\rr)$ be the BMO($\rr$)-closure of $C_c^\infty(\rr)$, which was introduced by
Neri \cite{N1975}; see also \cite{U1978}. Based on Theorem \ref{thm1}, we also obtain
the compactness characterization of the commutator $[b, C_\Gamma]$ in terms of
$\CMO(\rr)$ functions, which is the third main result of this article.

\begin{theorem}\label{thm2}
Let $b \in \BMO(\rr)$. Then the following statements are mutually equivalent:
\begin{itemize}
  \item [{\rm (i)}] $b \in \CMO(\rr)$.
  \item [{\rm (ii)}] The commutator $[b,C_{\Gamma}]$ is compact on $\morrey(\rr)$ for any $p \in (1,\infty)$  and $\lambda\in(0,1)$.
  \item [{\rm (iii)}] The commutator $[b,C_{\Gamma}]$ is compact on $\morrey(\rr)$ for some $p \in (1,\infty)$  and $\lambda\in(0,1)$.
\end{itemize}
\end{theorem}

An outline of this article is in order. Section \ref{s2} is divided into two subsections.
We first give the proof of Theorem \ref{thm1} in Subsection \ref{s2.1} and then the proof of
Theorem \ref{fac} in Subsection \ref{s2.2}. In the proof of Theorem \ref{thm1},
inspired by \cite{LOR2017}, we obtain an  auxiliary result suitable for $C_\Gamma$
(see Lemma \ref{product} below), which is on the domination of the local mean oscillation of $b$
on a given interval $I$ by the difference of $|b(x)-b(y)|$ pointwise on subsets of
$I\times 5I$, where, for any given interval $I:=I(x, r)$ with $x\in\rr$ and $r\in(0,\infty)$,
$5I:=I(x, 5r)$. Compared with the argument used in the proof of \cite[Proposition 3.1]{LOR2017}
therein, the argument used here is simpler due to the specific structure of the kernel $C_\Gamma$.

Section \ref{s3} is devoted to the proof of Theorem \ref{thm2} and is splitted into two subsections.
We mention that the implication relation from (ii) to (iii) of Theorem \ref{thm2} is obvious.
In Subsection \ref{s3.1}, by using a variant of the Fr\'{e}chet-Kolmogorov theorem, which
was obtained by  \cite[Theorem 1.12]{CDW2012} and suitable for $\morrey(\rr)$, and via establishing
the boundedness of the maximal operator of the truncated Cauchy integral on $\morrey(\rr)$,
we show the implication relation from (i) to (ii) of Theorem \ref{thm2}.
On the other hand, in Subsection \ref{s3.2}, we first obtain a lemma for the upper and the lower
bounds of integrals of $[b, C_\Gamma]f_j$ on certain intervals, for $b\in\BMO(\rr)$ and proper
function $f_j$. Using this  and a contradiction argument via an equivalent characterization
of CMO$(\rr)$ established by Uchiyama \cite{U1978}, we give the proof of the implication
relation from (iii) to (i) of Theorem \ref{thm2}.

Finally, we make some conventions on notation. Throughout the article,
we denote by $C$ and $\widetilde{C}$ {positive constants} which
are independent of the main parameters, but they may vary from line to
line. Constants with subscripts, such as $C_0$ and $A_1$, do
not change in different occurrences. Moreover, the symbol $f\lesssim g$ represents that
$f\le Cg$ for some positive constant $C$. If $f\lesssim g$ and $g\lesssim f$,
we then write $f\sim g$.

\section{Proofs of Theorems \ref{thm1} and \ref{fac}}\label{s2}

This section is divided into two subsections. We first present the proof of
Theorem \ref{thm1} in Subsection \ref{s2.1}
and, as an application of Theorem \ref{thm1}, we further give the proof of
Theorem \ref{fac} in Subsection \ref{s2.2}.

\subsection{Proof of Theorem \ref{thm1}}\label{s2.1}
Inspired by the recent work \cite{LOR2017}, we show Theorem \ref{thm1}
by means of the so-called local mean oscillations of functions. Recall that,
given a measurable function $f$ on $\rr$ and an interval $I\subset\rr$, the
\emph{local mean oscillation} $\omega_\mu(f;I)$ of $f$ on $I$ is defined by setting
$$\omega_\mu(f;I):=\inf_{c\in\rr}\lf[(f-c)\chi_I\r]^*(\mu|I|),$$
where $\mu\in(0,1)$ and $f^*$ denotes the non-increasing rearrangement of $f$, i.\,e.,
for any $t\in[0,\infty)$,
$$f^*(t):=\inf\lf\{\alpha\in(0,\infty):\,\,|\{x\in\rr:|f(x)|>\alpha\}|<t\r\}.$$

\begin{lemma}[\cite{LOR2017}, Lemma 2.1]\label{lmo}
  Let $\mu\in(0,\frac{1}{8}]$ and $f$ be a measurable function on $\rr$. If
  $$\sup_{I\subset\rr} \omega_\mu(f;I)<\infty,$$ then
  $f\in \BMO(\rr)$ and $$\|f\|_{\BMO(\rr)}\lesssim\sup_{I\subset\rr} \omega_\mu(f;I).$$
\end{lemma}

The following technical lemma is a variant of \cite[Proposition 3.1]{LOR2017}, which is
suitable for the Cauchy integral $C_\Gamma$.

\begin{lemma}\label{product} Let $b\in L^1_{\rm loc}(\rr)$.
  Then, for any interval $I\subset\rr$, there exist measurable sets $E\subset I$ and $F\subset5I$
  such that
  \begin{enumerate}
    \item[{\rm(i)}] if\ \ $G:=E\times F$, then $|G|\sim|I|^2;$
    \item[{\rm(ii)}] $\omega_{\frac{1}{8}}(b;I)\le|b(x)-b(y)|, \,\, \forall(x,y)\in E\times F;$
    \item[{\rm(iii)}] $b(x)-b(y)$ does not change sign in $E\times F$.
  \end{enumerate}
\end{lemma}
\begin{proof}
We first recall the median value as in \cite{J1983}.
  For any $b\in L_{\mathrm{loc}}^1(\rr)$ and interval $I\subset\rr$, let $\alpha_I(b)$ be a
  real number such that
  $$\inf_{c\in\rr} \frac{1}{|I|} \int_I |b(x)-c|\,dx$$
  is attained. Observe that $\alpha_I(b)$ exists and may not be unique. Moreover,
  $\alpha_I(b)$ satisfies that
  \begin{align}\label{median}
  \lf|\{x\in I:\,\,b(x)>\alpha_I(b)\}\r|\le \frac{1}{2}|I|\quad \mathrm{and}\quad
    |\{x\in I:\,\,b(x)<\alpha_I(b)\}|\le \frac{1}{2}|I|.
  \end{align}

  For $I:=I(x_0,r)$ with $x_0\in\rr$ and $r\in(0,\infty)$, let $\widetilde{I}:=(x_0+4r,r)$.
  Then $\widetilde{I}\subset5I$ and
  $x-y<0$ for any $x\in I$, $y\in\widetilde{I}$. To show Lemma \ref{product}(ii),
  we first prove that there exists $\mathcal{E}\subset I$, $|\mathcal{E}|=\frac{1}{8}|I|$
  such that, for any $x\in\mathcal{E}$,
\begin{align}\label{w1/8}
\omega_{\frac{1}{8}}(b;I)\le\lf|b(x)-\alpha_{\widetilde{I}}(b)\r|.
\end{align}
Noticing
$$\omega_{\frac{1}{8}}(b;I)=\inf_{c\in\rr}\lf[(b-c)\chi_I\r]^*\lf(\frac{1}{8}|I|\r)
\le\lf\{\lf[b-\alpha_{\widetilde{I}}(b)\r]\chi_I\r\}^*\lf(\frac{1}{8}|I|\r)=:t_0,$$
we claim that, for any $t\le t_0$,    
\begin{align}\label{1/8|I|}
\lf|\lf\{x\in I:\,\,\lf|b(x)-\alpha_{\widetilde{I}}(b)\r|\geq t\r\} \r|\geq\frac{1}{8}|I|.
\end{align}
  Indeed, if $t<t_0$, then \eqref{1/8|I|} holds true trivially. For $t=t_0$,
  taking $\{t_j\}_{j\in\nn}\subset(0,t_0)$ and $\{t_j\}_{j\in\nn}\uparrow t_0$, we then have
 \begin{align}
    \lf|\lf\{x\in I:\,\,\lf|b(x)-\alpha_{\widetilde{I}}(b)\r|\geq t_0\r\}\r|
    &=\lf|\bigcap_{j\in\nn}\lf\{x\in I:\,\,\lf|b(x)-\alpha_{\widetilde{I}}(b)\r|\geq t_j\r\}\r|\notag\\
    &=\lim_{j\to\infty}\lf|\lf\{x\in I:\,\,\lf|b(x)-\alpha_{\widetilde{I}}(b)\r|\geq t_j\r\}\r|
    \geq\frac{1}{8}|I|\notag
  \end{align}
  and hence \eqref{1/8|I|} holds true as well. This finishes the proof of the above claim \eqref{1/8|I|}.
  Now, for $t=\omega_{\frac{1}{8}}(b;I)$, take
  $\mathcal{E}\subset\{x\in I:\,\,|b(x)-\alpha_{\widetilde{I}}(b)|\geq t\}$ satisfying $|\mathcal{E}|=\frac{1}{8}|I|$ and hence \eqref{w1/8} holds true.

  Next, we come to show that there exist $E\subset\mathcal{E}$, $F\subset\widetilde{I}$ such that
  $|E|=\frac{1}{16}|I|$, $|F|=\frac{1}{2}|\widetilde{I}|$,
  \begin{align}\label{|bx-median|<|bx-by|}
    |b(x)-\alpha_{\widetilde{I}}(b)|\le|b(x)-b(y)|,\,\,\forall(x,y)\in E\times F
  \end{align}
  and
  \begin{align}\label{bx-byNCS}
    b(x)-b(y)\,\,\mathrm{does\ not\ change\ sign\ in}\,\, E\times F.
  \end{align}
  Indeed, let
  $$E_1:=\lf\{x\in\mathcal{E}:\,\,b(x)\geq\alpha_{\widetilde{I}}(b)\r\}\quad\mathrm{and}\quad E_2:=\lf\{x\in\mathcal{E}:\,\,b(x)\le\alpha_{\widetilde{I}}(b)\r\};$$
  $$F_1:=\lf\{y\in\widetilde{I}:\,\,b(y)\le\alpha_{\widetilde{I}}(b)\r\}\quad\mathrm{and}\quad F_2:=\lf\{y\in\widetilde{I}:\,\,b(y)\geq\alpha_{\widetilde{I}}(b)\r\}.$$
  Then $|F_1|\geq\frac{1}{2}|\widetilde{I}|$, $|F_2|\geq\frac{1}{2}|\widetilde{I}|$ and
  there exists $i\in\{1,2\}$ such that $|E_i|\geq\frac{1}{2}|\mathcal{E}|$.
  Without loss of generality, we may assume $|E_1|\geq\frac{1}{2}|\mathcal{E}|$.
  Then there exist $E\subset E_1$ and $F\subset F_1$ such that
  $|E|=\frac{1}{2}|\mathcal{E}|$ and $|F|=\frac{1}{2}|\widetilde{I}|$.
  Thus,
  $$\lf|b(x)-\alpha_{\widetilde{I}}(b)\r|=b(x)-\alpha_{\widetilde{I}}(b)\le b(x)-b(y),\,\, \forall(x,y)\in E\times F.$$
  Therefore, \eqref{|bx-median|<|bx-by|} and \eqref{bx-byNCS} hold true.

  By \eqref{w1/8} and \eqref{|bx-median|<|bx-by|}, we know that Lemma \ref{product}(ii) holds true.
  Meanwhile, Lemma \ref{product}(iii) follows from \eqref{bx-byNCS}. In addition, $|G|=|E||F|=\frac{1}{32}|I|^2$, namely, Lemma \ref{product}(i) holds true.
  This finishes the proof of Lemma \ref{product}.
\end{proof}

Now, we give the proof of Theorem \ref{thm1}.

\begin{proof}[Proof of Theorem \ref{thm1}]
Assume that (i) holds true first, that is, $b\in \BMO(\rr)$.
Since $C_\Gamma$ is a standard Calder\'on-Zygmund operator and the commutator $[b, C_\Gamma]$
is bounded on $L^p(\rr)$ for any $p\in(1, \infty)$ (see \cite[Lemma 3.3 and Theorem 1.1]{LNWW2017}),
from \cite[Theorems 2.1 and 2.2]{FLY1999} (see also \cite[Theorem 3.3]{KS2009}),
we deduce that, for any $p\in(1,\infty)$ and $\lambda\in(0,1)$, $[b, C_\Gamma]$ is also bounded on $\morrey(\rr)$ and
$$\lf\|[b,C_{\Gamma}]\r\|_{\morrey(\rr)\to\morrey(\rr)}\ls\|b\|_{\BMO(\rr)}.$$
This implies (ii).

Since the implication relation from (ii) to (iii) is obvious, it follows that,
to show Theorem \ref{thm1}, it suffices to show that (iii) implies (i).
To this end, assume that $[b, C_\Gamma]$ is bounded on $\morrey(\rr)$ for some
$p\in(1,\infty)$ and $\lambda\in(0,1)$. To show that $b\in \BMO(\rr)$,
for any interval $I\subset\rr$, let $G$ be as in Lemma \ref{product}(i).
Then $|x-y|\sim|I|$ for all $(x,y)\in G$. By Lemma \ref{product}(ii), we have
  $$\omega_{\frac{1}{8}}(b;I)|G|=\iint_G\omega_{\frac{1}{8}}(b;I)\,dx\,dy\le\iint_G|b(x)-b(y)|\,dx\,dy
  \lesssim|I|\iint_G|b(x)-b(y)|\frac{1}{|x-y|}\,dx\,dy.$$
  By Lemma \ref{product}, the definition of $C_\Gamma$ and the H\"{o}lder inequality, we conclude that
  \begin{align}
    \omega_{\frac{1}{8}}(b;I)&\lesssim\frac{|I|}{|G|}\iint_G|b(x)-b(y)|\frac{1}{|x-y|}\,dx\,dy
    \lesssim\frac{1}{|I|}\int_E\lf|\int_F[b(x)-b(y)]\frac{1}{x-y}\,dy\r|\,dx\notag\\
    &\sim\frac{1}{|I|}\int_E\lf|\int_F[b(x)-b(y)]\frac{x-y}{(x-y)^2+[A(x)-A(y)]^2}\,dy\r|\,dx\notag\\
    &\lesssim\frac{1}{|I|}\int_E\lf|\int_F[b(x)-b(y)]\frac{1}{x-y+i[A(x)-A(y)]}\,dy\r|\,dx\notag\\
    &\sim\frac{1}{|I|}\int_E\lf|[b,C_\Gamma]\chi_F(x)\r|\,dx
    \lesssim\lf[\frac{1}{|I|}\int_E\lf|[b,C_\Gamma]\chi_F(x)\r|^p\,dx\r]^{\frac{1}{p}}
    \lesssim|I|^{-\frac{1-\lambda}{p}}\|[b,C_\Gamma]\chi_F\|_{\morrey(\rr)}\notag\\
    &\lesssim|I|^{-\frac{1-\lambda}{p}}\|\chi_F\|_{\morrey(\rr)}\|[b,C_\Gamma]\|_
    {\morrey(\rr)\to\morrey(\rr)}
    \lesssim|I|^{-\frac{1-\lambda}{p}}\|\chi_{\widetilde{I}}\|_{\morrey(\rr)}\|[b,C_\Gamma]\|_
    {\morrey(\rr)\to\morrey(\rr)}\notag\\
    &\sim|I|^{-\frac{1-\lambda}{p}}\lf|\widetilde{I}\r|^{\frac{1-\lambda}{p}}\|[b,C_\Gamma]\|_
    {\morrey(\rr)\to\morrey(\rr)}
    \sim\|[b,C_\Gamma]\|_{\morrey(\rr)\to\morrey(\rr)}.\notag
  \end{align}
  From Lemma \ref{lmo}, we deduce that $b\in\BMO(\rr)$ and
  $$\|b\|_{\BMO(\rr)}\lesssim\|[b,C_\Gamma]\|_{\morrey(\rr)\to\morrey(\rr)}.$$
Thus, (i) holds true.  This finishes the proof of Theorem \ref{thm1}.
\end{proof}

\subsection{Proof of Theorem \ref{fac}}\label{s2.2}

This subsection is devoted to the proof of Theorem \ref{fac}. We begin with the following lemma
established in \cite{KM2006}.

\begin{lemma}\label{H1}
 If $\int_{\rr}f(x)\,dx=0$ and, for some $N\in(10,\infty)$ and $|x_0-y_0|=N$,
 $$|f(x)|\le\lf[\chi_{I(x_0,\,1)}(x)+ \chi_{I(y_0,\,1)}(x)\r]$$ for any $x\in\rr$,
 then there exists a positive constant $C$, independent of $N$, $x_0$ and $y_0$, such that
 $$\|f\|_{H^1(\rr)}\le C \log N.$$
\end{lemma}

\begin{proof}[Proof of Theorem \ref{fac}]
For any $b\in\BMO(\rr)$, by the John-Nirenberg inequality, we know that $b\in L_{\loc}^p(\rr)$ for any
$p\in(1,\infty)$. For any $g,\,h\in L_c^\infty(\rr)$, from
$[h^{\lambda,\,p'}(\rr)]^\ast=\morrey(\rr)$ and Theorem \ref{thm1}, it follows that
 \begin{align*}
 &\lf|\int_{\rr}b(x)\lf[g(x)C_\Gamma^*h(x)-h(x)C_\Gamma g(x)\r]\,dx\r|\\
 &\quad=\lf|\int_{\rr}\lf[h(x)C_\Gamma(gb)(x)-b(x)h(x)C_\Gamma g(x)\r]\,dx\r|\\
 &\quad=\lf|\int_{I}h(x)[b,C_\Gamma]g(x)\,dx\r|
 \lesssim\|h\|_{h^{\lambda,\,p'}(\rr)}\|[b,C_\Gamma]g\|_{\morrey(\rr)}\\
 &\quad\lesssim\|b\|_{\BMO(\rr)}\|h\|_{h^{\lambda,\,p'}(\rr)}\|g\|_{\morrey(\rr)}.
 \end{align*}
By the duality theorem between $\CMO(\rr)$ and $H^1(\rr)$ (\cite[Theorem 4.1]{CW1977}), we obtain
$$gC_\Gamma^*h-hC_\Gamma g\in H^1(\rr)\quad {\rm and}\quad\lf\|gC_\Gamma^*h-hC_\Gamma g\r\|_{H^1(\rr)}\lesssim\|h\|_{h^{\lambda,\,p'}(\rr)}\|g\|_{\morrey(\rr)}.$$
Now, suppose
$$f=\sum_{l=1}^\infty\sum_{k=1}^\infty \lambda_k^l\lf(g_k^l C_\Gamma^\ast h_k^l-h_k^l C_\Gamma g_k^l\r)$$
 as in Theorem \ref{fac}.
Then, by the arguments as above, we have $f\in H^1(\rr)$ and
$$\|f\|_{H^1(\rr)}\lesssim\sum_{l=1}^\infty\sum_{k=1}^\infty\lf|\lambda_k^l\r|\lf\|g_k^l\r\|_{\morrey(\rr)}
 \lf\|h_k^l\r\|_{h^{\lambda,\,p'}(\rr)}.$$
Thus,
$$\|f\|_{H^1(\rr)}\lesssim\inf\lf\{\sum_{l=1}^\infty\sum_{k=1}^\infty\lf|\lambda_k^l\r|
  \lf\|g_k^l\r\|_{\morrey(\rr)}\lf\|h_k^l\r\|_{h^{\lambda,\,p'}(\rr)}
  :\,\,f=\sum_{l=1}^\infty\sum_{k=1}^\infty \lambda_k^l\lf(g_k^l C_\Gamma^\ast h_k^l-h_k^l C_\Gamma g_k^l\r)\r\}.$$

To prove the converse, let $a$ be an $H^1(\rr)$-atom, supported in $I(x_0,r)$
with $x_0\in\rr$ and $r\in(0,\infty)$, satisfying that
$$\|a\|_{L^\infty(\rr)}\le r^{-1} \quad \mathrm{and}\quad \int_{\rr}a(x)\,dx=0.$$
Choose an integer $N$ sufficiently large which we shall determine later,
and select $y_0\in \rr$ such that $|x_0-y_0|=Nr$.
Now, for any $x\in\rr$, let
$$g(x):=\chi_{I(y_0,\,r)}(x)\quad \mathrm{and}\quad h(x):=\frac{-a(x)}{C_\Gamma g(x_0)}.$$
Since $C_\Gamma$ satisfies the $m$-$n$-homogeneous condition with $m=n=1$ (\cite[Lemma 3.4]{LNWW2017}),
it follows that
\begin{equation}\label{1nh}
  |C_\Gamma g(x_0)|\geq C_1 N^{-1}
\end{equation}
for some positive constant $C_1$.
By some routine calculations, we obtain
$$\|g\|_{\morrey(\rr)}\le |I(y_0,r)|^{-(\lambda-1)/p}\quad \mathrm{and}\quad
\|h\|_{h^{\lambda,p'}(\rr)}\le C_1 N \lf|I(x_0,r)\r|^{(\lambda-1)/p}.$$
Thus,
$$\|g\|_{\morrey(\rr)} \|h\|_{h^{\lambda,p'}(\rr)}\le C_1 N.$$
Moreover, it is easy to show that
$$\int_{\rr}\lf[g(x)C_\Gamma^*h(x)-h(x)C_\Gamma g(x)\r]dx=0.$$
Consequently, we have
\begin{equation}\label{cancel}
  \int_{\rr}\lf\{a(x)-\lf[g(x)C_\Gamma^*h(x)-h(x)C_\Gamma g(x)\r]\r\}\,dx=0.
\end{equation}
To apply Lemma \ref{H1}, we claim that the following two inequalities hold true:
$$|C_\Gamma g(x_0)-C_\Gamma g(x)|\lesssim N^{-2},\,\,\forall x\in I(x_0,r)$$
and
$$|C_\Gamma^*h(x)|\lesssim N^{-1}r^{-1},\,\,\forall x\in I(y_0,r).$$
Indeed, since $C_\Gamma$ is a standard Calder\'on-Zygmund kernel, we deduce that, for any $x\in I(x_0,r)$,
 $$|C_\Gamma g(x_0)-C_\Gamma g(x)|
 =\lf| \int_{I(y_0,r)}[C_\Gamma(x_0,y)-C_\Gamma(x,y)]\,dy \r|
 \lesssim  \int_{I(y_0,r)}\frac{|x_0-x|}{|x-y|^2} \,dy
 \lesssim  N^{-2}$$
 and, by \eqref{1nh} as well as the cancellation moment condition of atoms, we obtain,
 for any $x\in I(y_0,r)$,
 \begin{align*}
   \lf|C_\Gamma^*h(x)\r|=&\lf| \frac{1}{C_\Gamma g(x_0)} \int_{\rr}C_\Gamma(y,x)[-a(y)]\,dy \r|
   \lesssim  N  \lf| \int_{\rr}C_\Gamma(y,x)a(y)\,dy \r| \\
   \lesssim& N \int_{I(x_0,\,r)}\lf|C_\Gamma(x_0,x)-C_\Gamma(y,x)\r|\|a\|_{L^\infty(\rr)}\,dy
   \lesssim N r \frac{|x-y|}{|x_0-x|^2} r^{-1}
   \lesssim N^{-1}r^{-1}.
 \end{align*}
From the above estimates, we deduce that, for any $x\in\rr$,
\begin{align}\label{kk}
 &\lf|a(x)-\lf[g(x)C_\Gamma^*h(x)-h(x)C_\Gamma g(x)\r]\r|\\
 &\quad=\lf| \frac{a(x)[C_\Gamma g(x_0)-C_\Gamma g(x)]}{C_\Gamma g(x_0)}-g(x)C_\Gamma^*h(x)\r| \notag \\
 &\quad\lesssim \lf\{\frac{|a(x)|}{|C_\Gamma g(x_0)|}\lf|C_\Gamma g(x_0)-C_\Gamma g(x)\r|+|g(x)|\lf|C_\Gamma^*h(x)\r|\r\} \notag \\
 &\quad\lesssim N^{-1}r^{-1} \lf[\chi_{I(x_0,\,r)}(x)+ \chi_{I(y_0,\,r)}(x)\r].\notag
\end{align}
By \eqref{cancel}, \eqref{kk} and Lemma \ref{H1}, we know that
$$\lf\|a-\lf(gC_\Gamma^*h-hC_\Gamma g\r)\r\|_{H^1(\rr)}\le C_2N^{-1} \log N$$
for some positive constant $C_2$ depending on $C_1$.

Next, for any $f\in H^1(\rr)$, we can write $f=\sum_{k=1}^\infty \lambda_k^1 a_k^1$ in ${H^1(\rr)}$
norm via the atomic decomposition, where $\{a_k^1\}_{k\in\nn}$ are ${H^1(\rr)}$-atoms and
$$\sum_{k=1}^\infty \lf|\lambda_k^1\r|\le C_3\lf\|f\r\|_{H^1(\rr)}$$
for some constant $C_3\in(1,\infty)$ independent of $f$.
Then there exist $\{g_k^1\}_{k\in\nn} \subset \morrey(\rr)$ and $\{h_k^1\}_{k\in\nn} \subset h^{\lambda,p'}(\rr)$ such that
$$\lf\|g_k^1\r\|_{\morrey(\rr)}\lf\|h_k^1\r\|_{h^{\lambda,p'}(\rr)}\le C_1N$$
and
$$\lf\|a_k^1-\lf(g_k^1C_\Gamma^*h_k^1-h_k^1C_\Gamma g_k^1\r)\r\|_{H^1(\rr)}\le C_2 N^{-1} \log N.$$
Now we write
$$f=\sum_{k=1}^\infty \lambda_k^1 a_k^1
=\sum_{k=1}^\infty \lambda_k^1\lf(g_k^1C_\Gamma^*h_k^1-h_k^1C_\Gamma g_k^1\r)
+\sum_{k=1}^\infty \lambda_k^1 \lf[a_k^1-\lf(g_k^1C_\Gamma^*h_k^1-h_k^1C_\Gamma g_k^1\r)\r]
=:\mathrm{M}_1+\mathrm{E}_1.$$
Choosing $N\in(10,\infty)$ sufficiently large such that $C_2 N^{-1} \log N\in(0,\frac{1}{2})$,
we then have
$$\sum_{k=1}^\infty\lf\|\lambda_k^1g_k^1\r\|_{\morrey(\rr)}\lf\|h_k^1\r\|_{h^{\lambda,p'}(\rr)}
\le C_1N\sum_{k=1}^\infty \lf|\lambda_k^1\r|\le C_{(N)}\lf\|f\r\|_{H^1(\rr)}$$
and
$$\lf\|\mathrm{E}_1\r\|_{H^1(\rr)}\le C_2 N^{-1} \log N\lf\|f\r\|_{H^1(\rr)} \le\frac{1}{2}\lf\|f\r\|_{H^1(\rr)},$$
where $C_{(N)}$ is a positive constant depending on $C_1$, $C_3$ and $N$, but being independent of $f$.

Since $\mathrm{E}_1\in H^1(\rr)$, for the above given $C_3$, there exist a sequence of atoms $\{a_k^2\}_{k\in\nn}$ and numbers $\{\lambda_k^2\}_{k\in\nn}$ such that
$\mathrm{E}_1=\sum_{k=1}^\infty \lambda_k^2 a_k^2$ and
$$\sum_{k=1}^\infty \lf|\lambda_k^2\r|\le C_3 \lf\|\mathrm{E}_1\r\|_{H^1(\rr)}.$$
Notice that the positive constant $C_1$ in \eqref{1nh} is determined uniformly by
the 1-1-homogeneous condition (\cite{LW2017}).
Thus, there exist $\{g_k^2\}_{k\in\nn} \subset \morrey(\rr)$ and $\{h_k^2\}_{k\in\nn} \subset h^{\lambda,\,p'}(\rr)$ such that
$$\lf\|g_k^2\r\|_{\morrey(\rr)}\lf\|h_k^2\r\|_{h^{\lambda,\,p'}(\rr)}\le C_1N$$
and
$$\lf\|a_k^2-\lf(g_k^2C_\Gamma^*h_k^2-h_k^2C_\Gamma g_k^2\r)\r\|_{H^1(\rr)}\le C_2 N^{-1} \log N$$
for the same positive constants $C_1$ and $C_2$ as above.
Next we write
$$\mathrm{E}_1=\sum_{k=1}^\infty \lambda_k^2 a_k^2
=\sum_{k=1}^\infty \lambda_k^2\lf(g_k^2C_\Gamma^*h_k^2-h_k^2C_\Gamma g_k^2\r)+
\sum_{k=1}^\infty \lambda_k^2 \lf[a_k^2-\lf(g_k^2C_\Gamma^*h_k^2-h_k^2C_\Gamma g_k^2\r)\r]
=:\mathrm{M}_2+\mathrm{E}_2.$$
Therefore,
$$f=\mathrm{M}_1+\mathrm{E}_1=\mathrm{M}_1+\mathrm{M}_2+\mathrm{E}_2=
\sum_{l=1}^2\sum_{k=1}^\infty\lambda_k^l\lf(g_k^lC_\Gamma^*h_k^l-h_k^lC_\Gamma g_k^l\r)+\mathrm{E}_2.$$
With the same choice of $N$ as above, we have
$$\sum_{k=1}^\infty\lf\|\lambda_k^2g_k^2\r\|_{\morrey(\rr)}\lf\|h_k^2\r\|_{h^{\lambda,\,p'}(\rr)}
\le C_1N\sum_{k=1}^\infty \lf|\lambda_k^2\r|\le C_{(N)}\lf\|\mathrm{E}_1\r\|_{H^1(\rr)}\le\frac{1}{2}C_{(N)}\lf\|f\r\|_{H^1(\rr)}$$
and
$$\lf\|\mathrm{E}_2\r\|_{H^1(\rr)}\le C_2 N^{-1} \log N\lf\|\mathrm{E}_1\r\|_{H^1(\rr)} \le\frac{1}{2}\lf\|\mathrm{E}_1\r\|_{H^1(\rr)}\le\frac{1}{2^2}\lf\|f\r\|_{H^1(\rr)}.
$$

Continuing in this way, we conclude that, for any $L\in\nn$, $f$ has the representation
$$f=\sum_{l=1}^L\sum_{k=1}^\infty\lambda_k^l\lf(g_k^lC_\Gamma^*h_k^l-h_k^lC_\Gamma g_k^l\r)
+\mathrm{E}_L$$
satisfying
$$\sum_{k=1}^\infty\lf\|\lambda_k^Lg_k^L\r\|_{\morrey(\rr)}\lf\|h_k^L\r\|_{h^{\lambda,\,p'}(\rr)}
\le C_1N\sum_{k=1}^\infty \lf|\lambda_k^L\r|\le \frac{1}{2^{L-1}}C_{(N)}\lf\|f\r\|_{H^1(\rr)}$$
and
$$\lf\|\mathrm{E}_L\r\|_{H^1(\rr)}\le C_2 N^{-1} \log N\lf\|\mathrm{E}_{L-1}\r\|_{H^1(\rr)} \le\frac{1}{2^{L}}\lf\|f\r\|_{H^1(\rr)}.
$$
Letting $L\to\infty$, we find that
$$\sum_{l=1}^\infty\sum_{k=1}^\infty\lf\|\lambda_k^lg_k^l\r\|_{\morrey(\rr)}
\lf\|h_k^l\r\|_{h^{\lambda,\,p'}(\rr)}\le 2C_{(N)}\lf\|f\r\|_{H^1(\rr)}$$
and
$$  f=\sum_{l=1}^\infty\sum_{k=1}^\infty \lf(g_k^lC_\Gamma^\ast h_k^l-h_k^lC_\Gamma g_k^l\r)\,\, \mathrm{in}\,\, H^1(\rr).
$$
This finishes the proof of Theorem \ref{fac}.
\end{proof}

\section{Proof of Theorem \ref{thm2}}\label{s3}

In this section, we present the proof of Theorem \ref{thm2} by splitting it into two subsections.
In Section \ref{s3.1}, we give the proof of the implication relation from (i) to (ii),
that is, we show that, if $b\in {\rm CMO}(\mathbb R)$, then the commutator $[b, C_\Gamma]$
is compact on $L^{p,\,\lambda}(\mathbb R)$ for any $p \in (1,\infty)$ and $\lambda\in(0,1)$.
In Section \ref{s3.2}, we give the proof of the implication relation from (iii) to (i),
that is, if $[b, C_\Gamma]$ is compact on $L^{p,\,\lambda}(\mathbb R)$ for some
$p \in (1,\infty)$ and $\lambda\in(0,1)$, then $b\in {\rm CMO}(\mathbb R)$.

\subsection{Proof of the implication relation from (i) to (ii) of Theorem \ref{thm2}}\label{s3.1}

To prove the implication relation from (i) to (ii) of Theorem \ref{thm2}, we begin with
the following two technical lemmas.
The first one is a variant of the Fr\'{e}chet-Kolmogorov theorem suitable for $\morrey(\rr)$
by \cite[Theorem 1.12]{CDW2012}; see also \cite[p.\,275, Theorem (Fr\'{e}chet-Kolmogorov)]{Yosida}.
\begin{lemma}\label{cpt}
Let $p \in [1,\infty)$ and $\lambda \in (0,1)$. Suppose the subset
$E\subset\morrey(\rr)$ satisfies the following conditions:
\begin{enumerate}
\item[\rm(i)] $E$ is uniformly bounded, i.\,e.,
$$\sup_{f\in E} \|f\|_{\morrey(\rr)}<\infty;$$
\item[\rm(ii)] $E$ is uniformly equicontinuous, i.\,e.,
$$\lim_{y\to0}\|f(\cdot +y)-f(\cdot)\|_{\morrey(\rr)}=0\quad
\text{uniformly\,\,for\,\,any}\,\, f\in E;$$
\item[\rm(iii)] $E$ is uniformly vanishes at infinity, i.\,e.,
$$\lim_{\alpha\to\infty}\lf\|f\chi_{\{x\in\rr:\ |x|\ge\alpha\}}\r\|_{\morrey(\rr)}=0\quad
\text{uniformly\,\,for\,\,any}\,\, f\in E.$$
\end{enumerate}
Then $E$ is relatively compact in $\morrey(\rr)$.
\end{lemma}

We also need the boundedness of the maximal operator $C_{\Gamma,\,*}$ of the truncated
Cauchy integral on $\morrey(\rr)$; see also \cite{FLY1999, KS2009}.
\begin{lemma}\label{T*}
  Let $p\in(1,\infty)$ and $\lambda\in(0,1)$. Then $C_{\Gamma,\,*}$ is bounded on $\morrey(\rr)$,
  where
  $C_{\Gamma,\,*}$ is defined by setting, for any $f\in\morrey(\rr)$ and $x\in\rr$,
  $$C_{\Gamma,\,*}f(x):=\sup_{t>0}\lf| \int_{|x-y|>t}C_\Gamma(x,y)f(y)\,dy \r|.$$
\end{lemma}
\begin{proof}
  It suffices to prove that there exists a positive constant $C$ such that,
  for any interval $I\subset\rr$,
  $$\lf[\frac{1}{|I|^\lambda}\int_I\lf|C_{\Gamma,\,*}f(x)\r|^p\,dx\r]^{1/p}
  \le C\|f\|_{\morrey(\rr)}.$$
  Fix an interval $I:=I(x_0,r)$ with $x_0\in\rr$ and $r\in(0,\infty)$,
  and write $f=f_1+f_2$ where $f_1:=f\chi_{2I}$ and $f_2:=f-f_1$. Then we have
  \begin{align}\label{Cgf}
    \lf[\frac{1}{|I|^\lambda}\int_I\lf|C_{\Gamma,\,*}f(x)\r|^p\,dx\r]^{1/p}
    &\le\lf[\frac{1}{|I|^\lambda}\int_I\lf|C_{\Gamma,\,*}f_1(x)\r|^p\,dx\r]^{1/p}+
    \lf[\frac{1}{|I|^\lambda}\int_I\lf|C_{\Gamma,\,*}f_2(x)\r|^p\,dx \r]^{1/p}\\
    &=:\mathrm{A}+\mathrm{B}.\notag
  \end{align}
  By \cite[p.\,102, Theorem 5.14]{Duo},
  we know that $C_{\Gamma,\,*}$ is bounded on $L^p(\rr)$ for any $p\in(1,\infty)$. Thus, we have
  $$\mathrm{A}\lesssim\lf[\frac{1}{|I|^\lambda}\int_{2I}|f(x)|^p\,dx\r]^{1/p}
  \lesssim\|f\|_{\morrey(\rr)}.$$
  On the other hand, notice that $|x_0-y|\lesssim|x-y|$ for any $x\in I$ and $y\in (\rr\backslash2I)$.
  Thus, for any $x\in I$, we have
  $$\lf|C_{\Gamma,\,*}f_2(x)\r|\lesssim\int_{\rr\setminus2I}\frac{|f(y)|}{|x_0-y|}\,dy
    \lesssim\int_{|x_0-y|>2r}\frac{|f(y)|}{|x_0-y|}\,dy.$$
  Moreover, combining this with the H\"{o}lder inequality, we conclude that
  \begin{align*}
     \mathrm{B}\lesssim&|I|^{(1-\lambda)/p}\int_{|x_0-y|>2r}\frac{|f(y)|}{|x_0-y|}\,dy
     \lesssim|I|^{(1-\lambda)/p}\sum_{j=1}^\infty\int_{2^jr<|x_0-y|\le2^{j+1}r}
     \frac{|f(y)|}{|x_0-y|}\,dy\\
     \lesssim&|I|^{(1-\lambda)/p}\sum_{j=1}^\infty\frac{1}{|2^jI|}\int_{2^{j+1}I}|f(y)|\,dy\\
     \lesssim&|I|^{(1-\lambda)/p}\sum_{j=1}^\infty\frac{1}{|2^jI|}
                             \lf[\int_{2^{j+1}I}|f(y)|^pdy\r]^{1/p}\lf|2^{j+1}I\r|^{1/p'}\\
     \lesssim&\|f\|_{\morrey(\rr)}\sum_{j=1}^\infty
                     \frac{|I|^{(1-\lambda)/p}}{|2^{j+1}I|^{(1-\lambda)/p}}
     \lesssim\|f\|_{\morrey(\rr)}.
  \end{align*}
   Applying the estimates of A and B to \eqref{Cgf}, we then complete
   the proof of Lemma \ref{T*}.
\end{proof}

Now, we come to prove the implication relation from (i) to (ii) of Theorem \ref{thm2}. To this end,
we first recall that, for any $f\in L^1_{\mathrm{loc}}(\rr)$,
its Hardy-Littlewood maximal function $\mathcal{M}f$ is defined by setting, for any $x\in\rr$,
$$\mathcal{M}f(x):=\sup_{\substack{I\subset\rr\\I\ni x}}\frac{1}{|I|}\int_I|f(y)|\,dy,$$
where the supremum is taken over all intervals $I$ of $\rr$ which contain $x$.

\begin{proof}[Proof of Theorem \ref{thm2}]
  $(\mathrm{i})\Longrightarrow (\mathrm{ii})$.

  Assume $b\in \CMO(\rr)$.  It follows that, for any $\epsilon\in(0,\infty)$, there exists
  $b_\epsilon\in C_c^\infty(\rr)$ such that
  $\|b-b_\epsilon\|_{\BMO(\rr)}<\epsilon.$
  By Theorem \ref{thm1}, we know that, for any $p\in(1, \infty)$ and $\lambda\in(0, 1)$,
  \begin{align*}
    \|[b,C_\Gamma]f-[b_\epsilon,C_\Gamma]f\|_{\morrey(\rr)}
    \lesssim&\|b-b_\epsilon\|_{\BMO(\rr)}\|f\|_{\morrey(\rr)}
    \lesssim\epsilon\|f\|_{\morrey(\rr)}.
  \end{align*}
  Therefore, it suffices to show that $[b,C_\Gamma]$ is a compact operator when
  $b\in C_c^\infty(\rr)$ (see \cite[p.\,278]{Yosida}).
  Equivalently, it suffices to show that $[b,C_\Gamma]E$ is relatively compact when
  $b\in C_c^\infty(\rr)$ and $E\subset\morrey(\rr)$ is bounded, i.\,e., to
  show that $[b,C_\Gamma]E$ satisfies conditions (i) through (iii) of Lemma \ref{cpt}.

  We first point out that, by Theorem \ref{thm1} and the fact that $b\in \BMO(\rr)$,
  $[b,C_\Gamma]$ is bounded on $\morrey(\rr)$, which implies that $[b,C_\Gamma]E$
  satisfies (i) of Lemma \ref{cpt}.

  Next, since $b\in C_c^\infty(\rr)$, it follows that there exists a positive constant
  $R$ such that $\supp (b)\subset I(0,R)$.
  Let $\alpha\in(2R,\infty)$ and $E_\alpha:=\{x\in\rr:\,\,|x|\le \alpha \}$. Then, 
  for any $y\in I(0,R)$ and $x\in E_\alpha^\complement:=\rn\setminus E_\alpha=\{x\in\rn:\ |x|>\alpha\}$,
  we have $|x-y|\sim |x|$. Thus, by the H\"{o}lder inequality, we find that
  \begin{align*}
    |[b,C_\Gamma]f(x)|&\le\int_{\rr}\lf|C_\Gamma(x,y)\r||b(x)-b(y)|\lf|f(y)\r|\,dy\\
    &\lesssim\|b'\|_{L^\infty(\rr)}\int_{I(0,\,R)}\frac{|f(y)|}{|x-y|}\,dy
    \lesssim\|b'\|_{L^\infty(\rr)}\frac{1}{|x|}R^{\frac{1}{p'}
    +\frac{\lambda}{p}}\|f\|_{\morrey(\rr)}.
  \end{align*}
  Therefore, for any fixed interval $I:=I(x_0,r)$ with $x_0\in\rr$ and $r\in(0,\infty)$,
  similarly to the proof of Lemma \ref{T*}, we have
  \begin{align*}
    \frac{1}{|I|^\lambda}\int_I\lf|[b,C_\Gamma]f(x)\chi_{E_\alpha^\complement}(x)\r|^p\,dx
    \lesssim&\frac{1}{|I|^\lambda}\int_{I\cap E_\alpha^\complement}
    \frac{1}{|x|^p}R^{p-1+\lambda}\|f\|_{\morrey(\rr)}^p\,dx\\
    \lesssim&R^{p-1+\lambda}\|f\|_{\morrey(\rr)}^p\sum_{j=0}^\infty
    \frac{1}{|I\cap(2^{j+1}E_\alpha)|^\lambda}
    \int_{I\cap(2^{j+1}E_\alpha\setminus2^jE_\alpha)}\frac{1}{|x|^p}\,dx\\
    \lesssim&R^{p-1+\lambda}\|f\|_{\morrey(\rr)}^p\sum_{j=0}^\infty
    \frac{|2^{j+1}E_\alpha|^{1-\lambda}}{(2^j\alpha)^p}
    \sim\lf(\frac{R}{\alpha}\r)^{p-1+\lambda}\|f\|_{\morrey(\rr)}^p,
  \end{align*}
  where $2^j E_\alpha:=\{x\in\rr:\,\,|x|\le 2^j\alpha \}.$
  Hence we draw the conclusion that
  $$\lf\|[b,C_\Gamma]f\chi_{E_\alpha^\complement}\r\|_{\morrey(\rr)}
  \lesssim\lf(\frac{R}{\alpha}\r)^{(p-1+\lambda)/p}\|f\|_{\morrey(\rr)}.$$
  Therefore, condition (iii) of Lemma \ref{cpt} holds true for $[b,C_\Gamma]E$ as $\alpha\to\infty$.

  It remains to prove that $[b,C_\Gamma]E$ also satisfies (ii) of Lemma \ref{cpt}. Let
  $\epsilon\in(0,\frac{1}{2})$ be a fixed positive constant and $z\in\rr$ small enough. Then, for any
  $x\in\rr$, we have
  \begin{align*}
    &[b,C_\Gamma]f(x)-[b,C_\Gamma]f(x+z)\\
    &\quad=\int_{\rr}C_\Gamma(x,y)[b(x)-b(y)]f(y)\,dy-\int_{\rr}C_\Gamma(x+z,y)[b(x+z)-b(y)]f(y)\,dy\\
    &\quad=\int_{|x-y|>\epsilon^{-1}|z|}C_\Gamma(x,y)[b(x)-b(x+z)]f(y)\,dy\\
     &\quad\quad+\int_{|x-y|>\epsilon^{-1}|z|}[C_\Gamma(x,y)-C_\Gamma(x+z,y)][b(x+z)-b(y)]f(y)\,dy\\
     &\quad\quad+\int_{|x-y|\le\epsilon^{-1}|z|}C_\Gamma(x,y)[b(x)-b(y)]f(y)\,dy\\
     &\quad\quad-\int_{|x-y|\le\epsilon^{-1}|z|}C_\Gamma(x+z,y)[b(x+z)-b(y)]f(y)\,dy
     =:\sum_{i=1}^4 \mathrm{L}_i(x).
  \end{align*}

   We start with $\mathrm{L}_1$. Observe first that $|\mathrm{L}_1(x)|\le|b(x)-b(x+z)|C_{\Gamma,\,*}f(x)$.
   Due to $b\in C_c^\infty(\rr)$, we have $|b(x)-b(x+z)|\le \|b'\|_{L^\infty(\rr)}|z|$.
   Letting $z$ small enough depending on $\epsilon$ such that $|b(x)-b(x+z)|\le\epsilon$,
   it follows from Lemma \ref{T*} that
   \begin{align}\label{L1}
     \|\mathrm{L}_1\|_{\morrey(\rr)}
     \le\lf\||b(\cdot)-b(\cdot+z)|C_{\Gamma,\,*}f\r\|_{\morrey(\rr)}
     \lesssim\epsilon\|f\|_{\morrey(\rr)}.
   \end{align}

  Since $\epsilon\in(0,\frac{1}{2})$, we deduce that
  $$|x-y|>\epsilon^{-1}|z|\Longrightarrow |z|<|x-y|\epsilon<\frac{|x-y|}{2}.$$
  From the smoothness condition of the kernel $C_\Gamma$, we deduce that,
  for any $x,\,y,\,z\in\rr$ such that $|z|\le\frac{1}{2}|y-x|$,
  $$|C_\Gamma(x,y)-C_\Gamma(x+z,y)|\lesssim\frac{|z|}{|x-y|^2}.$$
  By this, together with $b\in C_c^\infty(\rr)$, we obtain
  \begin{align*}
    |\mathrm{L}_2(x)|&\lesssim|z|\int_{|x-y|>\epsilon^{-1}|z|}\frac{|f(y)|}{|x-y|^2}\,dy
    \sim|z|\sum_{k=0}^\infty\int_{2^k\epsilon^{-1}|z|<|x-y|\le2^{k+1}\epsilon^{-1}|z|}
    \frac{|f(y)|}{|x-y|^2}\,dy\\
    &\lesssim|z|\sum_{k=0}^\infty\frac{1}{(2^k\epsilon^{-1}|z|)^2}\int_{|x-y|\le2^{k+1}
    \epsilon^{-1}|z|}|f(y)|\,dy\\
    &\lesssim\sum_{k=0}^\infty\frac{\epsilon}{2^k}\frac{1}{I(x,2^{k+1}\epsilon^{-1}|z|)}
    \int_{|x-y|\le2^{k+1}\epsilon^{-1}|z|}|f(y)|\,dy
    \lesssim\epsilon \mathcal{M}f(x),
  \end{align*}
  where $\mathcal{M}f$ denotes the Hardy-Littlewood maximal function of $f$. Recalling that $\mathcal{M}$
  is bounded on Morrey spaces (see \cite{CF1987}), we arrive at the conclusion that
  \begin{align}\label{L2}
    \|\mathrm{L}_2\|_{\morrey(\rr)}\lesssim\epsilon\|f\|_{\morrey(\rr)}.
  \end{align}

  Observing that $C_\Gamma$ is a standard Calder\'{o}n-Zygmund kernel and $b\in C_c^\infty(\rr)$,
  by the mean value theorem, we find that, for any $x\in\rr$,
  \begin{align*}
    |\mathrm{L}_3(x)|&\lesssim\int_{|x-y|\le\epsilon^{-1}|z|}\frac{|x-y|}{|I(x,|x-y|)|}|f(y)|\,dy\notag\\
    &\sim\sum_{k=-\infty}^{-1}\int_{2^k\epsilon^{-1}|z|<|x-y|\le2^{k+1}\epsilon^{-1}|z|}
    \frac{|x-y|}{|I(x,|x-y|)|}|f(y)|\,dy\notag\\
    &\lesssim\sum_{k=-\infty}^{-1}2^k\epsilon^{-1}|z|\frac{1}{|I(x,2^{k+1}\epsilon^{-1}|z|)|}
    \int_{|x-y|\le2^{k+1}\epsilon^{-1}|z|}|f(y)|\,dy
    \lesssim\epsilon^{-1}|z|\mathcal{M}f(x).
  \end{align*}
  Therefore, we are led to the conclusion that
  \begin{align}\label{L3}
    \|\mathrm{L}_3\|_{\morrey(\rr)}\lesssim\epsilon\|f\|_{\morrey(\rr)}.
  \end{align}

  Similarly to the estimation of $\mathrm{L}_3$, we also have
 \begin{align}\label{L4}
    \|\mathrm{L}_4\|_{\morrey(\rr)}\lesssim\epsilon\|f\|_{\morrey(\rr)}.
  \end{align}

  Combining the estimates \eqref{L1} through \eqref{L4}, we obtain,
  for any $z\in\rr$ small enough,
  $$\lf\|[b,C_\Gamma]f(\cdot)-[b,C_\Gamma]f(\cdot+z)\r\|_{\morrey(\rr)}
  \lesssim\sum_{i=1}^4\|\mathrm{L}_i\|_{\morrey(\rr)}
  \lesssim\epsilon\|f\|_{\morrey(\rr)}.$$
  This shows that $[b,C_\Gamma]E$  satisfies (ii) of Lemma \ref{cpt}. Hence, $[b,C_\Gamma]$ is a
  compact operator. This finishes the proof of the implication from (i) to (ii) of Theorem \ref{thm2}.
\end{proof}

\subsection{Proof of the implication relation from (iii) to (i) of Theorem \ref{thm2}}\label{s3.2}

Since the implication relation from (ii) to (iii) of Theorem \ref{thm2} is obvious, it suffices to show the implication relation from (iii) to (i) of Theorem \ref{thm2}. To this end,
we first recall the following equivalent characterization of $\CMO(\rr)$ proved by
 Uchiyama \cite{U1978} (see also \cite{D2002}).

\begin{lemma}\label{CMO}
  A function $f\in\CMO(\rr)$ if and only if $f$ satisfies the following three conditions.
  \begin{enumerate}
    \item[{\rm(i)}] $\lim_{\delta\to0}\sup_{\{I\subset\rr:\,|I|<\delta\}}M(f,I)=0;$
    \item[{\rm(ii)}] $\lim_{R\to\infty}\sup_{\{I\subset\rr:\,|I|>R\}}M(f,I)=0;$
    \item[{\rm(iii)}] $\lim_{R\to\infty}\sup_{\{I\subset\rr:\,I\cap I(0,\,R)=\emptyset\}}M(f,I)=0.$
  \end{enumerate}
\end{lemma}

The following lemma gives the upper and the lower bounds of the integrals of
$\{[b,C_{\Gamma}]f_j\}_{j\in\nn}$
on certain intervals, resembling \cite[Lemma 4.1]{LNWW2017}. The proofs are similar, the major
change being the substitution of the power for $\{f_j\}_{j\in\nn}$. In what follows, for any
$r\in\rr$, we use $\lfloor r\rfloor$ to denote the largest integer not greater than $r$.

\begin{lemma}\label{L&U}
 Assume that $b\in\BMO(\rr)$ with $\|b\|_{\BMO(\rr)}=1$, and there exist $\delta\in(0,\infty)$ and a sequence $\{I_j\}_{j\in\nn}:=\{I(x_j,r_j)\}_{j\in\nn}$ of intervals, with $\{x_i\}_{j\in\nn}\subset\rr$
 and $\{r_i\}_{j\in\nn}\subset(0,\infty)$, such that, for any $j\in\nn$,
 \begin{align}\label{MbI}
 M(b,I_j):=\frac{1}{|I_j|}\int_{I_j}\lf|b(y)-b_{I_j}\r|\,dy>\delta.
 \end{align}
 Then there exist functions $\{f_j\}_{j\in\nn}\subset\morrey(\rr)$, positive constants $A_1$, $\widetilde{C_0}$, $\widetilde{C_1}$
 and $\widetilde{C_2}$ such that, for any $j\in\nn$ and integer
 $k\geq\lfloor \log_2 A_1\rfloor,\ \|f_j\|_{\morrey(\rr)}\le\widetilde{C_0}$,
 \begin{align}\label{LB}
  \int_{I_j^k}\lf|[b,C_{\Gamma}]f_j(y)\r|^p\,dy
  \geq\widetilde{C_1}\delta^p\frac{|I_j|^{p-1+\lambda}}{|2^k I_j|^{p-1}}
 \end{align}
 and
 \begin{align}\label{UB}
  \int_{2^{k+1}I_j\setminus2^kI_j}\lf|[b,C_{\Gamma}]f_j(y)\r|^pdy
  \le\widetilde{C_2}\frac{|I_j|^{p-1+\lambda}}{|2^k I_j|^{p-1}},
\end{align}
where $I_j^k:=(x_j+2^kr_j,x_j+2^{k+1}r_j)$.
\end{lemma}
\begin{proof}
  By Definition \ref{BMO} and the choice of $\alpha_I(b)$ in Lemma \ref{product},
  it is straightforward to show that, for any interval $I\subset \rr$,
  $$M(b,I)=\frac{1}{|I|}\int_I\lf|b(x)-b_I\r|\,dx\sim\frac{1}{|I|}\int_I\lf|b(x)-\alpha_I(b)\r|\,dx.$$
  Now, for any $j\in\nn$, we define the function $f_j$ as follows:
  $$f_j^1:=\chi_{I_{j,\,1}}-\chi_{I_{j,\,2}}:=\chi_{\{x\in I_j:\,\,b(x)>\alpha_{I_j}(b)\}}
    -\chi_{\{x\in I_j:\,\,b(x)<\alpha_{I_j}(b)\}},\quad
  f_j^2:=a_j \chi_{I_j}$$
  and
  $$f_j:=\lf|I_j\r|^{-(1-\lambda)/p}\lf(f_j^1-f_j^2\r),$$
  where $I_j$ is as in the assumption of Lemma \ref{L&U} and $a_j$ is a constant such that
  \begin{align}\label{intfj=0}
    \int_{\rr}f_j(x)\,dx=0.
  \end{align}
  Then, by the definition of $a_j$ and \eqref{median}, we claim that
  \begin{align}\label{1/2}
  \lf|a_j\r|\in\lf[0,\frac{1}{2}\r].
  \end{align}
  Indeed, for any $j\in\nn$, we have
  \begin{align*}
    0=\int_{\rr}f_j(x)\,dx=&\int_{\rr}\lf|I_j\r|^{-(1-\lambda)/p}\lf[\chi_{I_{j,\,1}}(x)-
                        \chi_{I_{j,\,2}}(x)-a_j \chi_{I_j}(x)\r]\,dx \\
                      =&\lf|I_j\r|^{-(1-\lambda)/p} \lf(\lf|I_{j,\,1}\r|-\lf|I_{j,\,2}\r|-a_j\lf|I_j\r|\r) \\
                   \le&\lf|I_j\r|^{-(1-\lambda)/p} \lf(\frac{1}{2}\lf|I_j\r|-\lf|I_{j,\,2}\r|-a_j\lf|I_j\r|\r)
                   \le\lf(\frac{1}{2}-a_j\r)\lf|I_j\r|^{-(1-\lambda)/p}\lf|I_j\r|.
  \end{align*}
  Thus, $a_j\in(-\infty,\frac{1}{2}]$. Similarly, $a_j\in[\frac{1}{2},\infty)$.
  Accordingly, $|a_j|\in[0,\frac{1}{2}]$.

  Besides, it follows immediately from the definition of $f_j$ and \eqref{1/2} that $\supp(f_j)\subset I_j$ and
  \begin{align}\label{>=0}
    f_j(y)\lf[b(y)-\alpha_{I_j}(b)\r]\geq 0, \,\, \forall y\in I_j.
  \end{align}
  Meanwhile, since \eqref{1/2} holds true, from a routine computation, we deduce that
  \begin{align}\label{|fj|}
    \lf|f_j(y)\r|\sim\lf|I_j\r|^{-(1-\lambda)/p}, \,\, \forall y\in \lf(I_{j,\,1}\bigcup I_{j,\,2}\r).
  \end{align}
  Moreover, by $\supp(f_j)\subset I_j$, we know that, for any interval $I\subset \rr$,
  $$\frac{1}{|I|^\lambda}\int_I\lf|f_j(x)\r|^p dx
       \lesssim\frac{1}{|I|^\lambda}\lf|I\cap I_j\r|\lf|I_j\r|^{-(1-\lambda)}
       \sim
       \begin{cases}
         \displaystyle{\frac{|I\cap I_j|}{|I_j|}\frac{|I_j|^\lambda}{|I|^\lambda},\quad |I|\geq|I_j|} \\
         \displaystyle{\frac{|I\cap I_j|}{|I|^\lambda}\frac{|I_j|^\lambda}{|I_j|},\quad |I|\le|I_j|}
       \end{cases}
       \lesssim 1.$$
  Thus,
  \begin{align*}
    \lf\|f_j\r\|_{\morrey(\rr)} \sim 1.
  \end{align*}

  Our task now is to prove inequality \eqref{LB} in Lemma \ref{L&U}. The trick of the proof is to
  notice the following decomposition: for any $y\in\rr$,
  \begin{align}\label{A-B}
    [b,C_\Gamma]f_j(y)=C_\Gamma\lf(\lf[b-\alpha_{I_j}(b)\r]f_j\r)(y)-
    \lf[b-\alpha_{I_j}(b)\r]C_\Gamma(f_j)(y)=:\mathrm{A}(y)-\mathrm{B}(y).
  \end{align}
  Fix a constant $A_1\in(4,\infty)$. Then, for any integer $k\geq\lfloor \log_2 A_1\rfloor$,
  we have
  \begin{align}\label{interval}
    2^{k+1}I_j\subset8I_j^k=\lf( x_j-\frac{5}{2}2^kr_j,x_j+\frac{11}{2}2^kr_j \r)
    \subset2^{k+3}I_j.
  \end{align}
  By \eqref{intfj=0}, \eqref{|fj|}, $\supp(f_j)\subset I_j$, the definition of $C_\Gamma$ and the fact that $|y-x_j|\geq2|z-x_j|$ for any $y\in(\rr\backslash2I_j)$ and $z\in I_j$, we conclude that,
  for any $y\in(\rr\backslash2I_j)$,
  \begin{align}\label{|B|}
     |\mathrm{B}(y)|&=\lf|\lf[b(y)-\alpha_{I_j}(b)\r]C_\Gamma(f_j)(y)\r|\\
       &\le\lf|b(y)-\alpha_{I_j}(b)\r|\int_{I_j}\lf|C_\Gamma(y,z)-C_\Gamma(y,x_j)\r|\lf|f_j(z)\r|dz \notag\\
       &\lesssim\lf|b(y)-\alpha_{I_j}(b)\r|\int_{I_j}\frac{|z-x_j|}{|y-x_j|^2}\lf|I_j\r|^{-(1-\lambda)/p}dz
       \sim \frac{|b(y)-\alpha_{I_j}(b)|}{|I_j|^{(1-\lambda)/p}|y-x_j|^2}\int_{I_j}\lf|z-x_j\r|dz \notag\\
       &\lesssim r_j\lf|I_j\r|^{1/p'+\lambda/p} \frac{|b(y)-\alpha_{I_j}(b)|}{|y-x_j|^2}.\notag
  \end{align}
  Notice that $I_j^k=(x_j+2^kr_j,x_j+2^{k+1}r_j)\subset (\rr\backslash2I_j)$ for any integer
  $k\geq\lfloor \log_2 A_1\rfloor$ and $A_1\in(4,\infty)$.
  Besides, $|x_j-y|\geq2^kr_j$ for all $y\in I_j^k$.
  Thus, from \eqref{|B|}, we deduce that
  \begin{align}\label{int|B|p}
    \int_{I_j^k}|\mathrm{B}(y)|^p\,dy\lesssim r_j^p\lf|I_j\r|^{p-1+\lambda}
                        \int_{I_j^k}\frac{|b(y)-\alpha_{I_j}(b)|^p}{|y-x_j|^{2p}}\,dy
                       \lesssim \frac{|I_j|^{p-1+\lambda}}{2^{2kp}r_j^p}
                        \int_{I_j^k}\lf|b(y)-\alpha_{I_j}(b)\r|^p\,dy.
  \end{align}
  Observe also that, by \eqref{interval}, we have $I_j^k\subset2^{k+1}I_j$
  for any integer $k\geq\lfloor \log_2 A_1\rfloor$ with $A_1\in(4,\infty)$.
  Thus, from $b\in \BMO(\rr)$ and the John-Nirenberg inequality, it follows that
  \begin{align*}
    \int_{I_j^k}\lf|b(y)-\alpha_{I_j}(b)\r|^p\,dy \notag
    \le&\int_{2^{k+1}I_j}\lf|b(y)-\alpha_{2^{k+1}I_j}(b)+\alpha_{2^{k+1}I_j}(b)-\alpha_{I_j}(b)\r|^p\,dy \notag \\
    \le&2^{p-1} \lf[\int_{2^{k+1}I_j}\lf|b(y)-\alpha_{2^{k+1}I_j}(b)\r|^p\,dy
    +2^{k+1}\lf|I_j\r|\lf|\alpha_{2^{k+1}I_j}(b)-\alpha_{I_j}(b)\r|^p \r] \notag \\
    \lesssim& 2^{p-1}\lf(2^{k+1}\lf|I_j\r|+k^p2^{k+1}\lf|I_j\r|\r)
    \lesssim k^p2^{k+1}\lf|I_j\r|.
  \end{align*}
  Therefore, we can estimate \eqref{int|B|p} ultimately by
  \begin{align}\label{int B}
  \int_{I_j^k}|\mathrm{B}(y)|^p\,dy\lesssim \frac{|I_j|^{p-1+\lambda}}{2^{2kp}r_j^p}k^p2^{k+1}\lf|I_j\r|
  \sim 2^{-kp}k^p\frac{|I_j|^{p-1+\lambda}}{|2^kI_j|^{p-1}}.
  \end{align}

  To estimate $\int_{I_j^k}|\mathrm{A}(y)|^p\,dy$, observe that $y>z$ for all
  $y\in I_j^k$ and $z\in I_j$, moreover,
  $$|y-z|\le \lf|x_j+2^{k+1}r_j-x_j+r_j\r|=\lf(2^{k+1}+1\r)r_j\le2^{k+2}r_j.$$
  Accordingly, using \eqref{MbI}, \eqref{>=0}, \eqref{|fj|} and the definition of $C_\Gamma$,
  we find that
  \begin{align*}
    |\mathrm{A}(y)|&=\lf| \int_{I_{j,\,1}\cup I_{j,\,2}} \frac{y-z-i[A(y)-A(z)]}{(y-z)^2+[A(y)-A(z)]^2}[b(z)-\alpha_{I_j}(b)]f_j(z)\,dz \r| \\
       &\gtrsim \int_{I_{j,\,1}\cup I_{j,\,2}} \frac{y-z}{(y-z)^2+[A(y)-A(z)]^2}\lf|b(z)-\alpha_{I_j}(b)\r|\lf|I_j\r|^{-(1-\lambda)/p}\,dz \\
       &\gtrsim \int_{I_{j,\,1}\cup I_{j,\,2}} \frac{1}{|y-z|}\lf|b(z)-\alpha_{I_j}(b)\r|\lf|I_j\r|^{-(1-\lambda)/p}\,dz \\
       &\gtrsim \frac{|I_j|^{-(1-\lambda)/p}}{2^{k+2}r_j}\int_{I_j}\lf|b(z)-\alpha_{I_j}(b)\r|\,dz
       \gtrsim \frac{|I_j|^{-(1-\lambda)/p}}{2^{k+2}r_j} M(b,I_j)\lf|I_j\r|
       \gtrsim \frac{\delta |I_j|^{(p-1+\lambda)/p}}{2^{k+2}r_j}.
  \end{align*}
  Thus,
  \begin{align}\label{int A}
  \int_{I_j^k}|\mathrm{A}(y)|^p\,dy\gtrsim \frac{\delta^p |I_j|^{p-1+\lambda}}{2^{p(k+2)}r_j^p} \lf|I_j^k\r|
  \sim \delta^p 2^{-(p+1)}\frac{|I_j|^{p-1+\lambda}}{|2^kI_j|^{p-1}}.
  \end{align}
  Applying \eqref{int B} and \eqref{int A} to \eqref{A-B}, we conclude that
  $$\int_{I_j^k}|[b,C_\Gamma]f(y)|^p\,dy
  \gtrsim \lf[\delta^p 2^{-(p+1)}-2^{-kp}k^p\r]\frac{|I_j|^{p-1+\lambda}}{|2^kI_j|^{p-1}}.$$
  Choose $A_1$ large enough such that, for any integer $k\geq\lfloor \log_2 A_1\rfloor$,
  $$\delta^p 2^{-(p+1)}-2^{-kp}k^p\gtrsim \delta^p.$$
  Thus, we obtain
  $$\int_{I_j^k}|[b,C_\Gamma]f(y)|^p\,dy\gtrsim \delta^p\frac{|I_j|^{p-1+\lambda}}{|2^kI_j|^{p-1}}.$$
  This shows that \eqref{LB} holds true.

  Next, we show that \eqref{UB} holds true. Observe that, for any $z\in I_j$ and $y\in I_j^k$,
  $$|y-z|\geq\lf|x_j+2^kr_j-x_j-r_j\r|=(2^k-1)r_j.$$
  We deduce the upper bound of $|A(y)|$ similarly. For any $y\in(\rr\setminus2I_j)$,
  by \eqref{|fj|}, we have
  \begin{align}\label{|A|}
    |\mathrm{A}(y)|&\le \int_{I_j}\lf|C_\Gamma(y,z)\r|\lf|b(z)-\alpha_{I_j}(b)\r|\lf|f_j(z)\r|\,dz\\
       &\lesssim \int_{I_j}\frac{1}{|y-z|}\lf|b(z)-\alpha_{I_j}(b)\r|\lf|I_j\r|^{-(1-\lambda)/p}\,dz
       \sim \frac{|I_j|^{-(1-\lambda)/p}}{(2^k-1)r_j}\int_{I_j}\lf|b(z)-\alpha_{I_j}(b)\r|\,dz\notag\\
       &\lesssim \frac{|I_j|^{-(1-\lambda)/p}}{(2^k-1)r_j}\|b\|_{\BMO(\rr)}\lf|I_j\r|
       \lesssim \frac{|I_j|^{(p-1+\lambda)/p}}{2^{k-1}r_j}.\notag
  \end{align}
  From \eqref{A-B}, \eqref{|B|} and \eqref{|A|}, we deduce that, for any integer
  $k\geq\lfloor \log_2 A_1\rfloor$,
  \begin{align*}
    \int_{2^{k+1}I_j\setminus2^kI_j}\lf|[b,C_{\Gamma}]f_j(y)\r|^p\,dy
    \lesssim&\int_{2^{k+1}I_j\setminus2^kI_j}|\mathrm{A}(y)|^pdy+
    \int_{2^{k+1}I_j\setminus2^kI_j}|\mathrm{B}(y)|^p\,dy \\
    \lesssim&\frac{|I_j|^{p-1+\lambda}}{2^{p(k-1)}r_j^p}2^k r_j
    +\frac{|I_j|^{p-1+\lambda}}{2^{2kp}r_j^p}\int_{2^{k+1}I_j\setminus2^kI_j}
    \lf|b(y)-\alpha_{I_j}(b)\r|^p\,dy\\
    \lesssim&\frac{2^{k+1}|I_j|^{p-1+\lambda}}{2^{p(k-1)}r_j^{p-1}}
    +\frac{|I_j|^{p-1+\lambda}}{2^{2kp}r_j^p}\lf|2^{k+1}I_j\r|k^p\|b\|_{\BMO(\rr)}^p\\
    \lesssim&\lf(2^{2p}+2^{p+1}\frac{k^p}{2^{kp}}\r)\frac{|I_j|^{p-1+\lambda}}{|2^kI_j|^{p-1}},
  \end{align*}
  which shows that \eqref{UB} holds true. This finishes the proof of Lemma \ref{L&U}.
\end{proof}

\begin{proof}[Proof of Theorem \ref{thm2}]
  $ (\mathrm{iii})\Longrightarrow (\mathrm{i})$.

  We employ the method from \cite{U1978}. Since $[b,C_\Gamma]$ is compact on $\morrey(\rr)$,
  it follows that $[b,C_\Gamma]$ is also bounded on $\morrey(\rr)$. Then, by Theorem \ref{thm1},
  we know that $b\in \BMO(\rr)$. To show $b\in \CMO(\rr)$, we use a contradiction argument
  via Lemma \ref{L&U}. Without loss of generality, we may assume that $\|b\|_{\BMO(\rr)}=1$.
  Notice that, if $b\notin\CMO(\rr)$, $b$ does not satisfy at least one of conditions (i) through (iii)
  of Lemma \ref{CMO}. We consider these cases orderly.

  \textbf{Case i)}
  Suppose $b$ does not satisfy (i) of Lemma \ref{CMO}, i.e.,
  $$\lim_{\delta\to 0}\sup_{\{I\subset\rr:\,|I|<\delta\}}M(b,I)=0,$$
  then there exist $\delta\in(0,\infty)$ and a sequence $\{I_j\}_{j\in\nn}$ of intervals satisfying
  $M(b,I_j)\in(\delta,\infty)$ for each $j\in\nn$ and $|I_j|\to0$ as $j\to\infty$.
  Let $f_j,\,\,\widetilde{C_1},\,\,\widetilde{C_2}$ and $A_1$ be as in Lemma \ref{L&U} and
  $A_2\in(A_1,\infty)$ large enough such that
  \begin{align}\label{A2}
    A_3:=8^{1-p}\widetilde{C_1}\delta^p A_1^{1-p}>
    \frac{2\widetilde{C_2}}{(1-2^{1-p})2^{\lfloor\log_2A_2\rfloor(p-1)}}.
  \end{align}
  Since $|I_j|\to0$ as $j\to\infty$, we may choose a subsequence $\{I_{j_l}^{(1)}\}_{l\in\nn}\subset\{I_j\}_{j\in\nn}$ such that
  \begin{align}\label{I1}
    \frac{|I_{j_{l+1}}^{(1)}|}{|I_{j_l}^{(1)}|}<\frac{1}{A_2}.
  \end{align}
  For fixed $l,\,m\in\nn$, let
    $$\mathcal{I}:=\lf(x_{j_l}^{(1)}+A_1r_{j_l}^{(1)},x_{j_l}^{(1)}+A_2r_{j_l}^{(1)}\r),\quad
    \mathcal{I}_1:=\mathcal{I}\setminus\lf\{y\in\rr:\,\,\lf|y-x_{j_{l+m}}^{(1)}\r|
    \le A_2r_{j_{l+m}}^{(1)} \r\}$$
    and
    $$\mathcal{I}_2:=\lf\{y\in\rr:\,\,\lf|y-x_{j_{l+m}}^{(1)}\r|> A_2r_{j_{l+m}}^{(1)}\r\}.$$
  Observe that
  $$\mathcal{I}_1\subset\lf\{y\in\rr:\,\,\lf|y-x_{j_l}^{(1)}\r|\le A_2r_{j_l}^{(1)} \r\}\cap\mathcal{I}_2
  \quad \mathrm{and}\quad \mathcal{I}_1=\mathcal{I}\setminus\lf(\mathcal{I}\setminus\mathcal{I}_2\r).$$
  Then, by the Minkowski inequality, we have
  \begin{align}\label{F1-F2}
    &\lf[\int_{A_2I_{j_l}^{(1)}}\lf|[b,C_\Gamma](f_{j_l})(x)
    -[b,C_\Gamma](f_{j_{l+m}})(x)\r|^p\,dx\r]^{1/p}\\
    &\quad\geq\lf[\int_{\mathcal{I}_1}\lf|[b,C_\Gamma](f_{j_l})(x)
    -[b,C_\Gamma](f_{j_{l+m}})(x)\r|^p\,dx\r]^{1/p}\notag\\
    &\quad\geq\lf[\int_{\mathcal{I}_1}\lf|[b,C_\Gamma](f_{j_l})(x)\r|^p\,dx\r]^{1/p}
    -\lf[\int_{\mathcal{I}_1}\lf|[b,C_\Gamma](f_{j_{l+m}})(x)\r|^p\,dx\r]^{1/p}\notag\\
    &\quad\geq\lf[\int_{\mathcal{I}_1}\lf|[b,C_\Gamma](f_{j_l})(x)\r|^p\,dx\r]^{1/p}
    -\lf[\int_{\mathcal{I}_2}\lf|[b,C_\Gamma](f_{j_{l+m}})(x)\r|^p\,dx\r]^{1/p}\notag\\
    &\quad=\lf[\int_{\mathcal{I}\setminus(\mathcal{I}\setminus\mathcal{I}_2)}
    \lf|[b,C_\Gamma](f_{j_l})(x)\r|^p\,dx\r]^{1/p}
    -\lf[\int_{\mathcal{I}_2}\lf|[b,C_\Gamma](f_{j_{l+m}})(x)\r|^p\,dx\r]^{1/p}
    =:\mathrm{F_1-F_2}.\notag
  \end{align}
  We first consider the term $\mathrm{F_1}$. To begin with, we now estimate the measure of
  $\mathcal{I}\setminus\mathcal{I}_2$. Assume that $E_{j_l}:=\mathcal{I}\setminus\mathcal{I}_2\neq\emptyset$.
  Then $E_{j_l}\subset A_2I_{j_{l+m}}^{(1)}$. Hence, by \eqref{I1}, we obtain
  \begin{align}\label{Eij}
    \lf|E_{j_l}\r|\le\lf|A_2I_{j_{l+m}}^{(1)}\r|=A_2\lf|I_{j_{l+m}}^{(1)}\r|<\lf|I_{j_l}^{(1)}\r|.
  \end{align}
  Now, let
  $$I_{j_l}^k:=\lf(x_{j_l}^{(1)}+2^k r_{j_l}^{(1)},x_{j_l}^{(1)}+2^{k+1} r_{j_l}^{(1)} \r),$$
  where integer $k\geq\lfloor \log_2 A_1\rfloor$ is as in Lemma \ref{L&U}. Then, by \eqref{Eij}, we have
  $$\lf|I_{j_l}^k\r|=2^kr_{j_l}^{(1)}=2^{k-1}\lf|I_{j_l}^{(1)}\r|\geq\lf|E_{j_l}\r|.$$
  From this, it follows that there exist at most two intervals, $I_{j_l}^{k_0}$ and
  $I_{j_l}^{k_0+1}$, such that $E_{j_l}\subset(I_{j_l}^{k_0}\cup I_{j_l}^{k_0+1})$.
  By \eqref{LB} in Lemma \ref{L&U}, we find that
  \begin{align*}
    \mathrm{F}_1^p&\geq\sum_{k=\lfloor\log_2 A_1\rfloor+1,\,k\neq k_0,\,k_0+1}^{\lfloor\log_2 A_2\rfloor}
    \int_{I_{j_l}^k}\lf|[b,C_\Gamma](f_{j_l})(y)\r|^p\,dy\notag\\
    &\geq\widetilde{C_1}\delta^p\sum_{k=\lfloor\log_2 A_1\rfloor+1,\,k\neq k_0,\,k_0+1}^
    {\lfloor\log_2 A_2\rfloor}\frac{|I_{j_l}^{(1)}|^{p-1+\lambda}}{|2^kI_{j_l}^{(1)}|^{p-1}}\notag\\
    &\geq\widetilde{C_1}\delta^p\sum_{k=\lfloor\log_2 A_1\rfloor+3}^{\lfloor\log_2 A_2\rfloor}
    \frac{1}{2^{k(p-1)}}\lf|I_{j_l}^{(1)}\r|^\lambda
    \geq A_3\lf|I_{j_l}^{(1)}\r|^\lambda.
  \end{align*}
  If $E_{j_l}:=\mathcal{I}\setminus\mathcal{I}_2=\emptyset$, the inequality above still holds true.

  On the other hand, from \eqref{A2} and \eqref{UB} in Lemma \ref{L&U}, we derive the estimate of
  $\mathrm{F}_2$ as follows:
  \begin{align*}
    \mathrm{F}_2^p&\le\sum_{k=\lfloor\log_2 A_2\rfloor}^\infty\int_{2^{k+1}I_{j_{l+m}}^{(1)}
    \setminus2^kI_{j_{l+m}}^{(1)}}\lf|[b,C_\Gamma](f_{j_{l+m}})(y)\r|^p\,dy\\
    &\le\widetilde{C_2}\sum_{k=\lfloor\log_2 A_2\rfloor}^\infty
    \frac{|I_{j_{l+m}}^{(1)}|^{p-1+\lambda}}{|2^kI_{j_{l+m}}^{(1)}|^{p-1}}
    \le\widetilde{C_2}\sum_{k=\lfloor\log_2 A_2\rfloor}^\infty
    \frac{1}{2^{k(p-1)}}\lf|I_{j_{l+m}}^{(1)}\r|^\lambda\\
    &\le\frac{\widetilde{C_2}}{(1-2^{1-p})2^{\lfloor\log_2A_2\rfloor(p-1)}}\lf|I_{j_{l+m}}^{(1)}\r|^\lambda
    =\frac{A_3}{2}\lf|I_{j_{l+m}}^{(1)}\r|^\lambda.
  \end{align*}

  Applying these two inequalities to \eqref{F1-F2}, we obtain
  $$\lf[\int_{A_2I_{j_l}^{(1)}}\lf|[b,C_\Gamma](f_{j_l})(x)-[b,C_\Gamma](f_{j_{l+m}})(x)\r|^p\,dx\r]^{1/p}
    \gtrsim \lf|I_{j_l}^{(1)}\r|^{\lambda/p},$$
  which leads to the conclusion that, for any $l,m\in\nn$,
  $$\lf\|[b,C_\Gamma](f_{j_l})-[b,C_\Gamma](f_{j_{l+m}})\r\|_{\morrey(\rr)}\gtrsim1.$$
  Therefore, $\{[b,C_\Gamma]f_j\}_{j\in\nn}$ is not relatively compact in $\morrey(\rr)$,
  which further implies that $[b,C_\Gamma]$ is not compact on $\morrey(\rr)$
  for any $p\in(1,\infty)$ and $\lambda\in(0,1)$.
  Accordingly, $b$ satisfies condition (i) of Lemma \ref{CMO}.

  \textbf{Case ii)}
  If $b$ violates condition (ii) of Lemma \ref{CMO}, that is
  $$\lim_{R\to\infty}\sup_{\{I\subset\rr:\,|I|>R\}}M(b,I)=0,$$
  then there exist $\delta\in(0,\infty)$ and a sequence $\{I_j\}_{j\in\nn}$ of intervals satisfying $M(b,I_j)\in(\delta,\infty)$ for each $j\in\nn$ and $|I_j|\to\infty$ as
  $j\to\infty$ as well.
  The proof of this case can be completed by a procedure analogous to that used in the proof of Case i).
  We take a subsequence $\{I_{j_l}^{(2)}\}_{l\in\nn}\subset\{I_j\}_{j\in\nn}$ such that
  \begin{align}\label{I2}
    \frac{|I_{j_l}^{(2)}|}{|I_{j_{l+1}}^{(2)}|}<\frac{1}{A_2}.
  \end{align}
  We can use a method similar to that used in the proof of Case i) and redefine our sets in a reversed order. That is, for fixed $l$ and $m$ belonging to $\nn$, let
    $$\mathcal{J}:=\lf(x_{j_{l+m}}^{(2)}+A_1r_{j_{l+m}}^{(2)},x_{j_{l+m}}^{(2)}+A_2r_{j_{l+m}}^{(2)}\r),
    \quad\mathcal{J}_1:=\mathcal{J}\setminus\lf\{y\in\rr:\,\,\lf|y-x_{j_{l}}^{(2)}\r|
    \le A_2r_{j_{l}}^{(2)} \r\}$$
    and
    $$\mathcal{J}_2:=\lf\{y\in\rr:\,\,\lf|y-x_{j_{l}}^{(2)}\r|> A_2r_{j_{l}}^{(2)}\r\}.$$
  Then we have
  $$\mathcal{J}_1\subset\lf\{y\in\rr:\,\,\lf|y-x_{j_{l+m}}^{(2)}\r|\le A_2r_{j_{l+m}}^{(2)}\r\}
  \cap\mathcal{J}_2 \quad \mathrm{and}\quad \mathcal{J}_1=\mathcal{J}\setminus(\mathcal{J}\setminus\mathcal{J}_2).$$
  As in Case i), by Lemma \ref{L&U} and \eqref{I2}, we know that $[b,C_\Gamma]$ is not compact on $\morrey(\rr)$, This contradiction implies that $b$ satisfies (ii) of Lemma \ref{CMO}.

  \textbf{Case iii)}
  Assume that condition (iii) of Lemma \ref{CMO} does not hold true for $b$. Namely,
  $$\lim_{R\to\infty}\sup_{\{I\subset\rr:\,I\cap I(0,\,R)=\emptyset\}}M(b,I)=0.$$
  Then there exists $\delta\in(0,\infty)$ such that, for any $R\in(0,\infty)$,
  there exists $I\subset[\rr\setminus I(0,R)]$ satisfying $M(b,I)\in(\delta,\infty)$.
  We claim that, for the above $\delta$, there exists a sequence $\{I_j\}_{j\in\nn}$ of
  intervals such that, for any $j\in\nn$,
  \begin{align}\label{MbIj}
    M(b,I_j)>\delta
  \end{align}
   and that, for any $l,\,\,m\in\nn$ and $l\neq m$,
   \begin{align*}
     A_2I_l\cap A_2I_m=\emptyset.
   \end{align*}
   Indeed, let $C_{(\delta)}$ be a positive constant which is determined later.
   Then, for any $R_1\in(C_{(\delta)},\infty)$, there exists an interval
   $I_1:=I(x_1,r_1)\subset\lf[\rr\setminus I(0,R_1)\r]$ such that \eqref{MbIj} holds true.
   Similarly, for
   $R_j:=|x_{j-1}|+2A_2C_{(\delta)}$ with $j\in\nn\setminus\{1\}$, there exists
   $I_j:=I(x_j,r_j)\subset[\rr\setminus I(0,R_j)]$
   satisfying \eqref{MbIj}. Repeating this procedure, we obtain $\{I_j\}_{j\in\nn}$ satisfying \eqref{MbIj}
   for each $j\in\nn$. Moreover, since $b$ satisfies condition (ii) of Lemma \ref{CMO},
   it follows that, for each aforementioned $\delta$, there exists a positive constant
   $\widetilde{C_{(\delta)}}$ such that $$M(b,I)<\delta$$ for any interval $I$ satisfying $|I|\in(\widetilde{C_{(\delta)}},\infty)$. This, together with the choice of $\{I_j\}_{j\in\nn}$, i.e.,
   \eqref{MbIj}, implies that, for any $j\in\nn$,
   $$r_j=\frac{1}{2}\lf|I_j\r|\le\frac{\widetilde{C_{(\delta)}}}{2}=:C_{(\delta)}.$$
   Notice also that, for any $j\in\nn$, $|x_{j+1}|-r_{j+1}\geq R_{j+1}=|x_j|+2A_2C_{(\delta)}$.
   Thus, for any $l$, $m\in\nn$, $l\neq m$, without loss of generality, we may assume $l<m$;
   we then have $|x_m-x_l|\geq|x_m|-|x_l|> 2A_2C_{(\delta)}$ and
   $$d\lf(A_2I_l, A_2I_m\r)\geq|x_l-x_m|-A_2r_l-A_2r_m
   >2A_2C_{(\delta)}-A_2C_{(\delta)}-A_2C_{(\delta)}=0.$$
   This implies that the above claim holds true.

   Now we define
   $$\mathcal{K}_1:=\lf(x_l+A_1r_l,x_l+A_2r_l\r)\quad\mathrm{and}\quad
    \mathcal{K}_2:=\lf\{y\in\rr:\,\,|y-x_{l+m}|>A_2r_{l+m}\r\}.$$
   Observe that $\mathcal{K}_1\subset\mathcal{K}_2$. Thus, as the estimations of $\mathrm{F}_1$
   and $\mathrm{F}_2$ in Case i), for any $l,\,m\in\nn$, $l\neq m$, we obtain
   \begin{align*}
     &\lf[\int_{A_2I_l}\lf|[b,C_\Gamma](f_l)(x)-[b,C_\Gamma](f_{l+m})(x)\r|^p\,dx\r]^{1/p}\\
     &\quad\geq\lf[\int_{\mathcal{K}_1}\lf|[b,C_\Gamma](f_l)(x)-[b,C_\Gamma](f_{l+m})(x)\r|^p\,dx\r]^{1/p}\\
     &\quad\geq\lf[\int_{\mathcal{K}_1}\lf|[b,C_\Gamma](f_l)(x)\r|^p\,dx\r]^{1/p}
     -\lf[\int_{\mathcal{K}_2}\lf|[b,C_\Gamma](f_{l+m})(x)\r|^p\,dx\r]^{1/p}
     \quad=:\mathrm{K}_1-\mathrm{K}_2.
   \end{align*}
   Again, by \eqref{LB} and \eqref{UB} in Lemma \ref{L&U}, as well as the definition of $A_3$ in
   \eqref{A2}, we conclude that $\mathrm{K}_1^p\geq A_3$ and $\mathrm{K}_2^p\le A_3/2$
   akin to Case i). To sum up, we obtain
   $$\lf[\int_{A_2I_l}\lf|[b,C_\Gamma](f_{l})(x)-[b,C_\Gamma](f_{l+m})(x)\r|^p\,dx\r]^{1/p}
    \gtrsim \lf|I_l\r|^{\lambda/p}.$$
   As a result,
   $$\lf\|[b,C_\Gamma](f_l)-[b,C_\Gamma](f_{l+m})\r\|_{\morrey(\rr)}\gtrsim1.$$
   This contradicts to the compactness of $[b, C_\Gamma]$ on $\morrey(\rr)$ for the given
   $p\in(1,\infty)$ and $\lambda\in(0,1)$, so $b$ also satisfies condition (iii) of Lemma \ref{CMO}.
   This finishes the proof of the implication relation from (iii) to (i) of Theorem \ref{thm2}
   and hence of Theorem \ref{thm2}.
\end{proof}

\bigskip

\noindent Jin Tao and Dachun Yang

\medskip

\noindent Laboratory of Mathematics and Complex Systems
(Ministry of Education of China),
School of Mathematical Sciences, Beijing Normal University,
Beijing 100875, People's Republic of China

\smallskip

\noindent{\it E-mails:} \texttt{jintao@mail.bnu.edu.cn} (J. Tao)

\hspace{0.86cm}\texttt{dcyang@bnu.edu.cn} (D. Yang)

\bigskip

\noindent

\noindent Dongyong Yang (Corresponding author)

\medskip

\noindent School of Mathematical Sciences, Xiamen University, Xiamen 361005, China

\smallskip

\noindent{\it E-mail:} \texttt{dyyang@xmu.edu.cn}

\end{document}